%% file: loti.tex
\documentclass[12pt]{amsart}

\usepackage[a4paper,margin=0.7in]{geometry}
\usepackage[foot]{amsaddr}
\usepackage{graphicx,color}
\usepackage{amsmath,amsthm,amsfonts,amssymb,mathrsfs}
\usepackage{cite}
\usepackage{changepage}

\begin{document}

\title{On the transfer of information in multiplier equations}

\author{Mirza Karamehmedovi\'c}
\address{mika@dtu.dk, Department of Applied Mathematics and Computer Science, Technical University of Denmark, DK-2800 Kgs. Lyngby, Denmark}

\author{David Winterrose}
\address{dawin@dtu.dk, Department of Applied Mathematics and Computer Science, Technical University of Denmark, DK-2800 Kgs. Lyngby, Denmark}

\maketitle

\begin{abstract}
We derive spectral width estimates for traces of tempered solutions of a large class of multiplier equations in $\mathbf{R}^n$. The estimates are uniform for solutions up to a given order. In the process, we find a rather explicit expression for a tempered fundamental solution of a multiplier. We successfully verify our spectral width estimates against numerical results in several scenarios involving the inhomogeneous Helmholtz equation in $\mathbf{R}^n$ with $n=1,\dots,9$. Our main result is directly applicable in the stability analysis of solutions of inverse source problems.
\end{abstract}

% \begin{keyword}
% Multipliers \sep Spectral analysis \sep Inverse source problem


% \end{keyword}

% \end{frontmatter}

% \newpage

\tableofcontents
% \linenumbers

\newtheorem{theorem}{Theorem}
\newtheorem{lemma}{Lemma}
\newtheorem{definition}{Definition}
\newtheorem{corollary}{Corollary}
\newtheorem{remark}{Remark}

\newcommand{\BW}{\mathscr{B}}
\newcommand{\singsupp}[1]{{\rm sing}\,\,{\rm supp}\,#1}
\newcommand{\supp}{{\rm supp}\,\,}
\newcommand{\dbar}{d\hspace*{-0.08em}\bar{}\hspace*{0.1em}}
\newcommand{\Op}{\textnormal{Op}}
\newcommand{\PW}{\textnormal{PW}} % Paley-Wiener space

\section{Introduction}\label{section:introduction}
Let $n\ge2$ be an integer, and write $\mathcal{F},\mathcal{F}^{-1}:\mathcal{S}'(\mathbf{R}^n)\rightarrow\mathcal{S}'(\mathbf{R}^n)$ for the Fourier transform and its inverse, respectively, as well as $\widehat{u}$ for $\mathcal{F}u$. % and $\mathbf{R}^n_{\ast}$ for $\mathbf{R}^n\setminus\{0\}$.
Also, write $\langle x\rangle=(1+|x|^2)^{1/2}$ for $x\in\mathbf{R}^n$. Assume $p\in C^{\infty}(\mathbf{R}^n)$ is an elliptic symbol of H\"ormander class $S^{\mu}(\mathbf{R}^n)$ for some real $\mu$, that is, assume that for each multi-index $\alpha\in\mathbf{N}_0^n$ there is a constant $C_{\alpha}$ satisfying
\[
\left|\partial^{\alpha}p(\xi)\right|\le C_{\alpha}\langle\xi\rangle^{\mu-|\alpha|},\quad\xi\in\mathbf{R}^n,
\]
as well as that there are constants $C$ and $R$ satisfying
\[
|p(\xi)|\ge C\langle\xi\rangle^{\mu},\quad\xi\in\mathbf{R}^n,\,\,|\xi|\ge R.
\]
Let $g\in C^{\infty}(\mathbf{R}\setminus\{0\})$ with $|g(\xi)|$ bounded from below by a positive constant; bounded from above as $|\xi|\rightarrow0$; and at most polynomially increasing as $|\xi|\rightarrow\infty$. Now assume the symbol $p$ is of the form
\begin{equation}\label{eqn:p}
p(\xi)=g(\xi)\prod_{j=1}^N(|\xi|-r_j)^{q_j},\quad\xi\in\mathbf{R}^n,
\end{equation}
where $N$ is a natural number, $r_j$ are positive constants satisfying $r_1<r_2<\cdots<r_N$, and $q_j$ are positive integer constants. Next, define the operator $P:\mathcal{S}'(\mathbf{R}^n)\rightarrow\mathcal{S}'(\mathbf{R}^n)$ by
\[
(Pu)(\phi)=(\mathcal{F}^{-1}p\widehat{u})(\phi)=\widehat{u}(p\mathcal{F}^{-1}\phi),\quad u\in\mathcal{S}'(\mathbf{R}^n),\,\,\phi\in\mathcal{S}(\mathbf{R}^n),
\]
that is, formally,
\[
Pu(x)=\mathcal{F}^{-1}(p\widehat{u})(x)=\int_{\xi\in\mathbf{R}^n}e^{ix.\xi}p(\xi)\widehat{u}(\xi)=\widehat{u}(e^{ix.(\cdot)}p(\cdot)),
\]
for $\quad u\in\mathcal{S}'(\mathbf{R}^n)$ and $x\in\mathbf{R}^n$. The function $p$ is called the symbol of $P$. Since $p$ depends on only the co-tangent variable $\xi$, it is a multiplier, and $P$ is a multiplier operator. It is evident that $P$ is a Fourier (frequency)-domain filter, transforming the spectrum of its argument $u$ by multiplying it with its symbol. Thus, in general, parts of the spectrum of $u$ may be augmented and other parts suppressed by the action of $P$.  Also, if an operator $F$ maps $Pu$ to the restriction of $u$ to a subset of $\mathbf{R}^n$, then it might be expected that $F$ essentially inverts the action of $P$ in the Fourier domain. This transformation of the Fourier spectrum of $u$ and of $Pu$ by the action of $P$ and $F$, respectively, is what we here mean by 'transfer of information in multiplier equations,' and the purpose of this work is to quantify, in a general setting, this transfer of information.

To state our main result, assume $u\in\mathcal{S}'(\mathbf{R}^n)$ and $f\in\mathcal{E}'(\mathbf{R}^n)$ satisfy
\begin{equation}\label{eqn:MAIN}
Pu=f\quad\text{in}\,\,\mathbf{R}^n.
\end{equation}
Fix a positive $R$ and orthonormal vectors $e_1,e_2$ in $\mathbf{R}^n$ such that the circle
\[
\mathcal{C}=\{e_1R\cos\theta+e_2R\sin\theta,\,\theta\in[0,2\pi]\}
\]
lies in the complement of the singular support of $f$, but not necessarily in the complement of the support of $f$. Figure~\ref{figure:Fig1} illustrates the setup.
\begin{figure}
\begin{center}
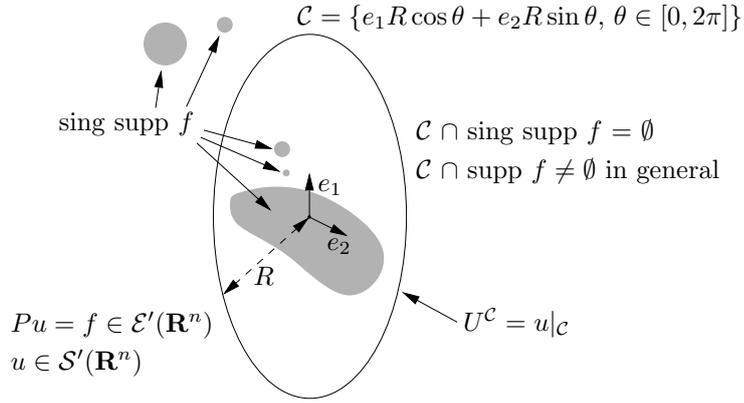
\caption{transfer of information from 'source' to 'measurement,' i.e., from the inhomogeneity $f$ to the trace $U^{\mathcal{C}}=u|_{\mathcal{C}}$.}\label{figure:Fig1}
\end{center}
\end{figure}
Since $p$ is elliptic and $\mathcal{C}$ is in the exterior of ${\rm sing}\,\,{\rm supp}\,f$, there exists a nonempty open neighborhood of $\mathcal{C}$ in $\mathbf{R}^n$ where $u$ is smooth. In particular, the 'measurement,' i.e., the trace $U^{\mathcal{C}}=u|_{\mathcal{C}}$, is a smooth function on $\mathcal{C}$, and its Fourier coefficients
\begin{equation}\label{eqn:FC}
\widehat{U}^{\mathcal{C}}_m=\int_{\theta=0}^{2\pi}e^{-im\theta}u(e_1R\cos\theta+e_2R\sin\theta),\quad m\in\mathbf{Z},
\end{equation}
are well-defined. Now fix $\chi\in C_0^{\infty}(\mathbf{R}^n)$ satisfying
\[
\chi(\xi)=\begin{cases}1,\quad&|\xi|\le 1,\\0,\quad&|\xi|\ge2.\end{cases}
\]
For any positive $\varrho$ define $\widehat{u}_{\varrho}=\chi(\cdot/\varrho)\widehat{u}\in\mathcal{E}'(\mathbf{R}^n)$ and let
\begin{equation}\label{eqn:FRT}
\widehat{U}^{\mathcal{C}}_{\varrho,m}=\int_{\theta=0}^{2\pi}e^{-im\theta}u_{\varrho}(e_1R\cos\theta+e_2R\sin\theta),\quad m\in\mathbf{Z},
\end{equation}
be the $m$'th Fourier coefficient of the trace of $u_{\varrho}$ on $\mathcal{C}$. The definition makes sense since, by Theorem 7.1.14 on page 165 in~\cite{HI}, we have $\widehat{\widehat{u}}_{\varrho}\in C^{\infty}(\mathbf{R}^n)$ with $\widehat{\widehat{u}}_{\varrho}(x)=(\widehat{u}_{\varrho})_{\xi}(e^{-ix.\xi})$ for $x\in\mathbf{R}^n$, and from~\cite[p. 164]{HI} we have $u_{\varrho}(\phi)=(2\pi)^{-n}\widehat{\widehat{u}}_{\varrho}(\phi(-\cdot))=(-2\pi)^{-n}(\widehat{\widehat{u}}_{\varrho}(-\cdot))(\phi)$ for every $\phi\in\mathcal{S}(\mathbf{R}^n)$, so $u_{\varrho}$ is well-defined pointwise with 
\begin{equation}\label{eqn:urho}
u_{\varrho}(x)=(-2\pi)^{-n}\widehat{\widehat{u}}_{\varrho}(-x)=(-2\pi)^{-n}(\widehat{u}_{\varrho})_{\xi}(e^{ix.\xi}),\quad x\in\mathbf{R}^n.
\end{equation}
In the following, $u\in\mathcal{S}'^d(\mathbf{R}^n)$, with $d\in\mathbf{N}_0$, if $u\in\mathcal{S}'(\mathbf{R}^n)$ and there is a constant $C$ satisfying
\begin{equation}\label{eqn:S'd}
|u(\phi)|\le C\sum_{|\alpha|\le d}\sup_{x\in\mathbf{R}^n}\langle x\rangle^d|\partial^{\alpha}\phi(x)|,\quad\phi\in\mathcal{S}(\mathbf{R}^n).
\end{equation}
The following two results combine to characterize the behavior of the discrete Fourier spectrum $(\widehat{U}^{\mathcal{C}}_m)_{m\in\mathbf{Z}}$ of the trace $U^{\mathcal{C}}$. Let $m$ be an integer and $u\in\mathcal{S}'(\mathbf{R}^n)$.
\begin{lemma}\label{lemma:convergence} $\lim_{\varrho\rightarrow\infty}\widehat{U}^{\mathcal{C}}_{\varrho,m}=\widehat{U}^{\mathcal{C}}_m$.
\end{lemma}
\begin{proof}
Since $\chi\in C_0^{\infty}(\mathbf{R}^n)$ and $\chi(0)=1$, we have $\lim_{\varrho\rightarrow\infty}\chi(\cdot/\varrho)\phi=\phi$ in $\mathcal{S}(\mathbf{R}^n)$ for every $\phi\in\mathcal{S}(\mathbf{R}^n)$. Hence $\lim_{\varrho\rightarrow\infty}\widehat{u}_{\varrho}=\lim_{\varrho\rightarrow\infty}\chi(\cdot/\varrho)\widehat{u}=\widehat{u}$ in $\mathcal{S}'(\mathbf{R}^n)$ with respect to its weak-$\ast$ topology. But $\mathcal{F}^{-1}$ is continuous on $\mathcal{S}'(\mathbf{R}^n)$, so $\lim_{\varrho\rightarrow\infty}u_{\varrho}=\lim_{\varrho\rightarrow\infty}\mathcal{F}^{-1}\widehat u_{\varrho}=u$ in $\mathcal{S}'(\mathbf{R}^n)$, and since $u_{\varrho}$ and $u$ are smooth in a neighborhood of $\mathcal{C}$, we have $\lim_{\varrho\rightarrow\infty}u_{\varrho}(x)=u(x)$ for every $x\in\mathcal{C}$.
\end{proof}
Now assume $u\in\mathcal{S}'^d(\mathbf{R}^n)$, $m\in\mathbf{Z}$, and write $\delta^{\rm Kr}$ for the Kronecker delta.
\begin{theorem}\label{thm:MAIN} There are constants $c'$ and, for any $\varrho>0$, $c_{\varrho}$ such that
\begin{align*}
\left|\widehat{U}^{\mathcal{C}}_{\varrho,m}\right|&\le c'+c_{\varrho}\sum_{j=1}^N\Bigl(\max\{1,|m|^{q_j},|m|^d\}|J_m(Rr_j)|\Bigr.\\&\Bigl.+\max\{1,|m|^{q_j-1},\delta^{\rm Kr}_{d\ge1}|m|^{d-1}\}|J_{m+1}(Rr_j)|\Bigr).
\end{align*}
\end{theorem}
\begin{remark}
In the absence of any auxiliary condition ensuring the uniqueness of solution of $Pu=f$, Theorem~\ref{thm:MAIN} distinguishes between classes of solutions according to their order $d$ as tempered distributions. In this sense, the assumption $u\in\mathcal{S}'^d(\mathbf{R}^n)$ may be regarded a weak substitute for a condition implying uniqueness. When the uniqueness of solution of $Pu=f$ is ensured, such as by the Sommerfeld radiation condition at infinity in the examples in sections~\ref{subsec:pointsource} and \ref{subsec:integrable}, the distributional order $d$ of $u\in\mathcal{S}'^d(\mathbf{R}^n)$ is determined by the problem dimension $n$ and the distributional order of $f$.
\end{remark}
\begin{remark}
Of all Euclidean spheres, only $S^0$, $S^1$ and $S^3$ admit a topological group structure~\cite{Megia-2007}. Therefore, it makes sense to define the Fourier transform of $u|_{S^{n-1}}$ only for $n=1$, $n=2$ and $n=4$. For other dimensions $n$, it is of course possible to pick specific bases of, say, $L^2(S^{n-1})$ and treat the projections of $u|_{S^{n-1}}$ onto the basis vectors as 'the Fourier coefficients of the measurement.' This approach, however, is rather arbitrary and we here choose to compute the Fourier coefficients of the measurement in terms of integrals over great circles for all dimensions $n$. Our analysis therefore estimates the magnitude of the spectral content of the measurement along any chosen direction in $S^{n-1}$.
\end{remark}
\begin{remark}
The main information contained in Theorem~\ref{thm:MAIN} is the expected upper bound on the spectral location $|m|$ of the onset of rapid decay of the Fourier coefficients $\widehat{U}^{\mathcal{C}}_{\varrho,m}$; call this spectral location the 'bandwidth' of the measurement $U^{\mathcal{C}}$. Theorem~\ref{thm:MAIN} allows us to estimate this bandwidth in terms of the behavior of the product $|m|^aJ_{m}$ with increasing $|m|$, which is independent of the parameter $\varrho$ in $\widehat{U}^{\mathcal{C}}_{\varrho,m}$, and in principle also independent of the dimension $n$ of the problem $Pu=f$. However, we show in Lemma~\ref{lemma:H} of Section~\ref{sec:ispHelm} that in certain cases additional, physically motivated conditions on the solution $u$ may make the bandwidth estimate dimension-dependent.
\end{remark}
\begin{remark}
The integral in~\eqref{eqn:FRT} is the Funk-Radon transform of the integrand, evaluated at a single chosen direction $\nu\in S^{n-1}$ orthogonal to the plane of $\mathcal{C}$.
\end{remark}

Theorem~\ref{thm:MAIN} holds for a large class of pseudodifferential operators $P$, in particular all such operators whose symbol $p(\xi)$ is independent of the base variable $x$ and that can be transformed by a diffeomorphic pullback to a symbol with a radially symmetric zero set, as in~\eqref{eqn:p}. A simple, but from a modeling perspective important, class of operators covered by our analysis are the differential operators of the form $P=\sum_{j=0}^Mc_j(-\Delta)^j$ with constants $c_j$ such that at least one of the zeros of the polynomial $t\mapsto\sum_{j=0}^Mc_jt^{2j}$ is positive.
More complicated variants may also be considered. For example, if  $A$ is a real symmetric positive-definite $n \times n$ matrix with spectral decomposition $A=Q^T \Lambda Q$, and $b\in \mathbf{R}^n$ and $c\in \mathbf{C}$ are constants, then the operator $P$ with symbol $p(\xi)=\sum_{j=0}^M c_j( \xi \cdot A \xi + b\cdot \xi + c)^j$ is amenable to the same type of analysis after pullback by the diffeomorphism $\Phi(r,\omega)=rQ^T\Lambda^{-\frac{1}{2}}\omega-\frac{1}{2}A^{-1}b$. The transformed symbol is $(p\circ \Phi)(r) =\sum_{j=0}^M c_j( r^2 - \frac{1}{4}b\cdot A^{-1}b+c)^j $, and we require this to have at least one positive root.
The simple case includes, e.g., the Helmholtz operator $P=\Delta+k^2$, $k>0$, that is used in the modeling of time-harmonic acoustic and electromagnetic waves, and that also appears in the time-independent Schr\"odinger equation with negative constant potential. Thus, one application of our work is in inverse source problems~\cite{Bao-2010,Griesmaier-2012,Griesmaier-2014,Griesmaier-2016,Griesmaier-2017,2018a-bandwidth,2018b-multi_freq_isp,2020a-Helmholtz_3D}.
Here, given a multiplier equation $Pu=f$, the 'source' $f$ is to be reconstructed from a 'measurement' of $u$, which is the trace of $u$ on a 'measurement set.' Quantifying the transfer of information between $u$ and $f$ allows us to estimate, e.g., which frequency components of the source cannot be reconstructed stably in the presence of measurement noise of some specified frequency and amplitude. See, e.g.,~\cite{Xu-2006,Pierri-2020,Pierri-2021} and references therein for instances of the spectral analysis of electromagnetic radiation operators and applications in, e.g., antenna design and measurements. References~\cite{Griesmaier-2012,Griesmaier-2014,Griesmaier-2016,Griesmaier-2017} characterize spectrally the \textit{far-field} Dirichlet trace of $u$ on $S^1$ and $S^2$ when $P$ is the Helmholtz operator on $\mathbf{R}^2$ and $\mathbf{R}^3$, respectively. In particular, in~\cite{Griesmaier-2017} and~\cite{Griesmaier-2014} the singular spectrum of the source-to-far-field operator (also called the restricted Fourier transform there) is specified as a low-pass filter. In~\cite{2018a-bandwidth} and~\cite{2020a-Helmholtz_3D} we characterize the spectrum of the near-field Dirichlet trace on $S^1$ and $S^2$ of solutions of the inhomogeneous Helmholtz equation in $\mathbf{R}^2$ and $\mathbf{R}^3$, respectively.

Section~\ref{sec:preparatory} contains definitions of particular distributions, as well as proofs of technical lemmas, needed for the proof of Theorem~\ref{thm:MAIN} in Section~\ref{sec:proof}. In particular, in Section~\ref{sec:preparatory} we find a rather explicit form of a fundamental solution $\mathcal{F}^{-1}\mathfrak{p}^{-1}\in\mathcal{S}'(\mathbf{R}^n)$ of $P$. In Section~\ref{sec:ispHelm} we test the spectral cutoff estimate computed using Theorem~\ref{thm:MAIN} against computed $\widehat{U}^{\mathcal{C}}_m$-spectra for the Helmholtz equation in $\mathbf{R}^n$, in several dimensions $n$ and for different types of 'sources' $f$. Finally, we summarize our results in Section~\ref{section:conclusion}.

\section{Preparatory definitions and results}\label{sec:preparatory}
Define
\[
\mathcal{T}_m=\int_{\theta=0}^{2\pi}e^{-i m\theta}e^{i a(X_1\cos\theta+X_2\sin\theta)},%\left(X_1\cos\theta+X_2\sin\theta\right)^k,
\]
where $a\in\mathbf{R}\setminus\{0\}$, $m\in\mathbf{Z}$, and where $X_1$ and $X_2$ are arbitrary complex constants. Writing $J_m$ for the Bessel function of the first kind and integer order $m$, we have
\begin{lemma}\label{lemma:integral}\[\mathcal{T}_m=\begin{cases}2\pi i^m J_m(a|X_1+i X_2|)\exp{-i m\angle(X_1+i X_2)},\quad&X_1+iX_2\neq0,\\2\pi\delta^{\rm Kr}_{m=0},\quad&X_1+iX_2=0.\end{cases}\]
\end{lemma}
\begin{proof}
Let $\mathcal{U}$ be the unit circle in the complex plane, centered at the origin. Substituting $z=e^{i\theta}$, writing $\zeta=X_1+i X_2$, and using that $\overline{z}=1/z$, we get
\begin{align*}
\mathcal{T}_m&=\oint_{z\in\mathcal{U}}\exp\left[\frac{i a}{2}\left(X_1(z+\overline{z})-i X_2(z-\overline{z})\right)\right]\frac{\overline{z}^m}{i z}\\&=-i\oint_{z\in\mathcal{U}}\exp\left[\frac{i a}{2}\left(z\overline{\zeta}+\zeta/z\right)\right]\frac{1}{z^{m+1}}\\&=-i\sum_{j=0}^{\infty}\sum_{k=0}^j\left(\frac{i a}{2}\right)^j\frac{1}{j!}{j\choose k}\overline{\zeta}^{k}\zeta^{j-k}\oint_{z\in\mathcal{U}}z^{2k-j-m-1}\\&=-i\sum_{j=0}^{\infty}\sum_{k=0}^j\left(\frac{i a}{2}\right)^j\frac{1}{j!}{j\choose k}\overline{\zeta}^{k}\zeta^{j-k}2\pi i\delta^{\rm Kr}_{2k-j-m-1=-1}.
\end{align*}
Thus, if $\zeta=0$ then $\mathcal{T}_m=-i(2\pi i\delta^{\rm Kr}_{m=0})=2\pi\delta^{\rm Kr}_{m=0}$. Assume in the following that $\zeta\neq0$. We can reduce the above double sum to the case $j=2k-m$. As $j\ge k$, we also have $k\ge m$. Thus
\begin{align*}
\mathcal{T}_m&=2\pi\sum_{k=m}^{\infty}\left(\frac{i a}{2}\right)^{2k-m}\frac{\overline{\zeta}^k\zeta^{k-m}}{k!(k-m)!}\\&=2\pi\left(\frac{2}{i a\zeta}\right)^m\sum_{k=m}^{\infty}\left(\frac{i a|\zeta|}{2}\right)^{2k}\frac{1}{k!(k-m)!}\\&=2\pi\left(\frac{2}{i a\zeta}\right)^m\cdot\left(\frac{i a|\zeta|}{2}\right)^mI_m(i a|\zeta|).
\end{align*}
Here, $I_m$ is the modified Bessel function of the first kind and order $m$. The result now follows from the relations $I_m(i x)=i^{-m}J_m(-x)$ and $J_m(-x)=(-1)^mJ_m(x)$, valid for $x\in\mathbf{R}\setminus\{0\}$ and all integer $m$.
\end{proof}

\begin{lemma}\label{lemma:diff_Bessel}
For every integer $m$ and every positive integer $j$ there is a constant $c$ satisfying
\[
\left|\partial_r^jJ_m(r)\right|\le c\left(|m|^j|J_m(r)|+|m|^{j-1}|J_{m+1}(r)|\right),\quad r>0.
\]
\end{lemma}
\begin{proof}
Write $\Pi_{\mu,\nu}(m,r)$ for any polynomial in $m$ and $r$, of degree at most $\mu$ in $m$ and at most $\nu$ in $r$. In particular, $\Pi_{-1,\nu}\equiv0$. If, for some $m\in\mathbf{Z}$, $j\in\mathbf{N}$, $\mu,\nu\in\mathbf{N}_0$ and all $r>0$, we have
\[
\partial_r^jJ_m(r)=r^{-j}\left(\Pi_{\mu,\nu}(m,r)J_m(r)+\Pi_{\mu-1,\nu}(m,r)J_{m+1}(r)\right),
\]
then
\begin{align*}
\partial_r^{j+1}J_m(r)&=-jr^{-j-1}\left(\Pi_{\mu,\nu}(m,r)J_m(r)+\Pi_{\mu-1,\nu}(m,r)J_{m+1}(r)\right)\\&+r^{-j}\Bigl[\Pi_{\mu,\nu-1}(m,r)J_m(r)+\Pi_{\mu,\nu}(m,r)(-J_{m+1}(r)+(m/r)J_m(r))\\&+\Pi_{\mu-1,\nu-1}(m,r)J_{m+1}(r)+\Pi_{\mu-1,\nu}(m,r)(-J_{m+2}(r)+\frac{m+1}{r}J_{m+1}(r))\Bigr]\\&=r^{-j-1}\Bigl[\Pi_{\mu+1,\nu}(m,r)J_m(r)+\Pi_{\mu,\nu+1}(m,r)J_{m+1}(r)\Bigr]\\&+r^{-j}\Pi_{\mu-1,\nu}(m,r)\left(J_m(r)-\frac{2(m+1)}{r}J_{m+1}(r)+\frac{m+1}{r}J_{m+1}(r)\right)\\&=r^{-j-1}\left[\Pi_{\mu+1,\nu+1}(m,r)J_m(r)+\Pi_{\mu,\nu+1}(m,r)J_{m+1}(r)\right],\quad r>0.
\end{align*}
We have here used the well-known recurrence relation
\[
J_m(r)+J_{m+2}(r)=\frac{2(m+1)}{r}J_{m+1}(r),
\]
as well as the well-known fact that
\[
\partial_rJ_m(r)=-J_{m+1}(r)+\frac{m}{r}J_m(r).
\]
But the last equality also implies that
\[
\partial_rJ_m(r)=r^{-1}\left[\Pi_{1,1}(m,r)J_m(r)+\Pi_{0,1}(m,r)J_{m+1}(r)\right],
\]
so
\[
\partial_r^jJ_m(r)=r^{-j}\left[\Pi_{j,j}(m,r)J_m(r)+\Pi_{j-1,j}(m,r)J_{m+1}(r)\right]
\]
for all $j\in\mathbf{N}_0$, $m\in\mathbf{Z}$ and positive $r$.
\end{proof}

Next, let $k$ be a positive integer. H\"ormander~\cite[(3.2.5) on p. 69; p. 71 bot.]{HI} defines the distributions $r_{\pm}^{-k}\in\mathcal{S}'(\mathbf{R})$ (written $x_{\pm}^{-k}$ in~\cite{HI}) by
\begin{equation}\label{eqn:rplus}
r_+^{-k}(\phi)=-\frac{1}{(k-1)!}\int_{r=0}^{\infty}(\ln r)\phi^{(k)}(r)+\frac{1}{(k-1)!}\phi^{(k-1)}(0)\sum_{j=1}^{k-1}\frac{1}{j},\quad\phi\in\mathcal{S}(\mathbf{R}),
\end{equation}
and
\begin{equation}\label{eqn:rminus}
r_-^{-k}(\phi)=r_+^{-k}(\phi(-\cdot))=-\frac{(-1)^k}{(k-1)!}\int_{r=0}^{\infty}(\ln r)\phi^{(k)}(-r)+\frac{(-1)^{k-1}}{(k-1)!}\phi^{(k-1)}(0)\sum_{j=1}^{k-1}\frac{1}{j}
\end{equation}
for $\phi\in\mathcal{S}(\mathbf{R})$.
These are extensions of the distributions
\[
\phi\mapsto\int_{r=0}^{\infty}r^a\phi(r)\quad\text{and}\quad\phi\mapsto\int_{r=-\infty}^{0}|r|^a\phi(r),
\]
respectively, from $\Re a>-1$ to $a=-k$. We here wish to extend the distributions
\[
\phi\mapsto\int_{r=0}^{\varrho}r^a\phi(r)\quad\text{and}\quad\phi\mapsto\int_{r=-\varrho}^{0}|r|^a\phi(r),\quad0<\varrho\le\infty,\,\,\,\phi\in\mathcal{S}(\mathbf{R}),
\]
from $\Re a>-1$ to $a=-k$. Our procedure is a slight generalization of~\cite[pp. 68--69]{HI}. Thus, fix a positive $\varrho$ and $\phi\in\mathcal{S}(\mathbf{R})$, and note that the function
\[
a\mapsto I_{a,\varrho}(\phi)=\int_{r=0}^{\varrho}r^a\phi(r)
\]
is holomorphic on $\{a\in\mathbf{C},\,\,\,\Re a>-1\}$; indeed, if $\Re a>-1$ then
\[
\frac{d}{da}I_{a,\varrho}(\phi)=\int_{r=0}^{\varrho}r^a(\ln r)\phi(r)
\]
is well-defined. If $a\in\mathbf{C}$ with $\Re a>0$ then we can integrate by parts once to get
\[
I_{a,\varrho}(\phi')=\phi(\varrho)\varrho^a-aI_{a-1,\varrho}(\phi),
\]
and if $\Re a>k-1$ then we can iterate the process to get
\[
I_{a,\varrho}(\phi^{(k)})=\sum_{j=0}^{k-1}(-1)^{k-1-j}\phi^{(j)}(\varrho)\varrho^{a-(k-1)+j}\prod_{\ell=0}^{k-2-j}(a-\ell)+(-1)^k\prod_{\ell=0}^{k-1}(a-\ell)I_{a-k,\varrho}(\phi),
\]
that is,
\begin{align*}
I_{a-k,\varrho}(\phi)&=\frac{(-1)^kI_{a,\varrho}(\phi^{(k)})-\sum_{j=0}^{k-1}(-1)^{j+1}\phi^{(j)}(\varrho)\varrho^{a-(k-1)+j}\prod_{\ell=0}^{k-j-2}(a-\ell)}{\prod_{\ell=0}^{k-1}(a-\ell)}\\
&=\frac{(-1)^kI_{a,\varrho}(\phi^{(k)})}{\prod_{\ell=0}^{k-1}(a-\ell)}+\sum_{j=0}^{k-1}\frac{(-1)^j\phi^{(j)}(\varrho)\varrho^{a-(k-1)+j}}{\prod_{\ell=k-j-1}^{k-1}(a-\ell)}.
\end{align*}
Equivalently, if $a\in\mathbf{C}$ with $\Re a>-1$, and $\phi\in\mathcal{S}(\mathbf{R})$, then
\[
I_{a,\varrho}(\phi)=\frac{(-1)^kI_{a+k,\varrho}(\phi^{(k)})}{\prod_{\ell=0}^{k-1}(a+k-\ell)}+\sum_{j=0}^{k-1}\frac{(-1)^j\phi^{(j)}(\varrho)\varrho^{a+j+1}}{\prod_{\ell=k-j-1}^{k-1}(a+k-\ell)}.
\]
The function $a\mapsto I_{a,\varrho}(\phi)$ is therefore holomorphic for $\Re a>-1-k$ with $a\neq-1,-2,\dots,-k$, and it has simple poles at $a=-1,-2,\dots,-k$.
The residue of $a\mapsto I_{a,\varrho}(\phi)$ at $-k$ is
\begin{align*}
\lim_{a\rightarrow-k}(a+k)I_{a,\varrho}(\phi)&=\lim_{a\rightarrow-k}\Biggl[(-1)^k\prod_{\ell=1}^{k-1}(a+k-\ell)^{-1}\int_{r=0}^{\varrho}r^{a+k}\phi^{(k)}(r)\Biggr.\\&+\sum_{j=0}^{k-1}(-1)^j\phi^{(j)}(\varrho)\varrho^{a+j+1}(a+k)\prod_{\ell=k-j-1}^{k-1}(a+k-\ell)^{-1}\Biggr]\\&=-\frac{1}{(k-1)!}\int_{r=0}^{\varrho}\phi^{(k)}(r)+\frac{\phi^{(k-1)}(\varrho)}{(k-1)!}=\frac{\phi^{(k-1)}(0)}{(k-1)!},
\end{align*}
that is, $\lim_{a\rightarrow k}(a+k)I_{a,\varrho}=(-1)^{k-1}\delta_0^{(k-1)}/(k-1)!$ in $\mathcal{S}'(\mathbf{R})$. With $\varepsilon=a+k$, we have
\begin{align*}
&\lim_{\varepsilon\rightarrow0}\left[I_{\varepsilon-k,\varrho}(\phi)-\frac{\phi^{(k-1)}(0)}{\varepsilon(k-1)!}\right]\\&=\lim_{\varepsilon\rightarrow0}\Biggl[\frac{(-1)^kI_{\varepsilon,\varrho}(\phi^{(k)})}{\prod_{\ell=0}^{k-1}(\varepsilon-\ell)}+\sum_{j=0}^{k-1}\frac{(-1)^j\phi^{(j)}(\varrho)\varrho^{\varepsilon-k+j+1}}{\prod_{\ell=k-j-1}^{k-1}(\varepsilon-\ell)}-\frac{\phi^{(k-1)}(0)}{\varepsilon(k-1)!}\Biggr]\\&=\lim_{\varepsilon\rightarrow0}\Biggl[\frac{(-1)^k}{\prod_{\ell=1}^{k-1}(\varepsilon-\ell)}\int_{r=0}^{\varrho}\frac{r^{\varepsilon}-1}{\varepsilon}\phi^{(k)}(r)+\frac{(-1)^k\int_{r=0}^{\varrho}\phi^{(k)}(r)}{\varepsilon\prod_{\ell=1}^{k-1}(\varepsilon-\ell)}-\frac{\phi^{(k-1)}(0)}{\varepsilon(k-1)!}\\&+\frac{(-1)^{k-1}\phi^{(k-1)}(\varrho)\varrho^{\varepsilon}}{\varepsilon\prod_{\ell=1}^{k-1}(\varepsilon-\ell)}+\sum_{j=0}^{k-2}\frac{(-1)^j\phi^{(j)}(\varrho)\varrho^{\varepsilon-k+j+1}}{\prod_{\ell=k-j-1}^{k-1}(\varepsilon-\ell)}\Biggr]\\&=\lim_{\varepsilon\rightarrow0}\Biggl[\frac{(-1)^k}{\prod_{\ell=1}^{k-1}(\varepsilon-\ell)}\int_{r=0}^{\varrho}\frac{r^{\varepsilon}-1}{\varepsilon}\phi^{(k)}(r)-\frac{\phi^{(k-1)}(0)}{\varepsilon}\left(\frac{1}{(k-1)!}+\frac{(-1)^k}{\prod_{\ell=1}^{k-1}(\varepsilon-\ell)}\right)\Biggr.\\&\Biggl.+\frac{(-1)^k\phi^{(k-1)}(\varrho)(1-\varrho^{\varepsilon})}{\varepsilon\prod_{\ell=1}^{k-1}(\varepsilon-\ell)}+\sum_{j=0}^{k-2}\frac{(-1)^j\phi^{(j)}(\varrho)\varrho^{\varepsilon-k+j+1}}{\prod_{\ell=k-j-1}^{k-1}(\varepsilon-\ell)}\Biggr]\\&=-\frac{1}{(k-1)!}\int_{r=0}^{\varrho}(\ln r)\phi^{(k)}(r)\\&+\frac{\phi^{(k-1)}(0)}{(k-1)!}\sum_{\nu=1}^{k-1}\frac{1}{\nu}-\frac{1}{(k-1)!}\sum_{j=0}^{k-2}\phi^{(j)}(\varrho)\varrho^{-k+j+1}(k-j-2)!.
\end{align*}
We therefore define the tempered distributions $r_{\pm,\varrho}^{-k}$, $k\in\mathbf{N}$, by
\begin{align*}
r_{+,\varrho}^{-k}(\phi)&=-\frac{1}{(k-1)!}\int_{r=0}^{\varrho}(\ln r)\phi^{(k)}(r)+\frac{\phi^{(k-1)}(0)}{(k-1)!}\sum_{\nu=1}^{k-1}\frac{1}{\nu}\\&-\frac{1}{(k-1)!}\sum_{j=0}^{k-2}\phi^{(j)}(\varrho)\varrho^{-k+j+1}(k-j-2)!,\quad\phi\in\mathcal{S}(\mathbf{R}),
\end{align*}
and
\begin{align*}
r_{-,\varrho}^{-k}(\phi)&=r_{+,\varrho}^{-k}(\phi(-\cdot))=-\frac{(-1)^k}{(k-1)!}\int_{r=0}^{\varrho}(\ln r)\phi^{(k)}(-r)+\frac{(-1)^{k-1}\phi^{(k-1)}(0)}{(k-1)!}\sum_{\nu=1}^{k-1}\frac{1}{\nu}\\&-\frac{1}{(k-1)!}\sum_{j=0}^{k-2}(-1)^j\phi^{(j)}(-\varrho)\varrho^{-k+j+1}(k-j-2)!,\quad\phi\in\mathcal{S}(\mathbf{R}),
\end{align*}
respectively. Clearly, the distributions $r_{\pm,\varrho}^{-k}$ specialize to H\"ormander's $r_{\pm}^{-k}$ when $\varrho=\infty$. Now for every real $b$ the mapping $\tau_b:\mathbf{R}\rightarrow\mathbf{R}$, $\tau_b(r)=r-b$, is smooth with surjective Jacobian $\tau_b'(r)=1$, so~\cite[Theorem 6.1.2]{HI} the pullback of $r_{\pm,\varrho}^{-k}$ by $\tau_b$ is given uniquely by
\[
(r-b)_{\pm,\varrho}^{-k}(\phi):=\tau_b^{\ast}r_{\pm,\varrho}^{-k}(\phi)=r_{\pm,\varrho}^{-k}(\phi(\cdot+b))=r_{\pm,\varrho}^{-k}(\phi\circ\tau_b^{-1}),\quad\phi\in\mathcal{S}(\mathbf{R}).
\]
Finally, writing $\xi=r\omega$ for $\xi\in\mathbf{R}^n$, with $r\ge0$ and $\omega\in S^{n-1}$, and letting $c_{jk}$ be the constants from the partial fraction decomposition
\[
\prod_{j=1}^N(r-r_j)^{-q_j}=\sum_{j=1}^N\sum_{k=1}^{q_j}c_{jk}(r-r_j)^{-k},\quad r\ge0,\,\,r\neq r_j,
\]
we define the tempered distribution
\begin{eqnarray}\label{eqn:p_inverse}
\mathfrak{p}^{-1}(\phi)&=&\sum_{j=1}^N\sum_{k=1}^{q_j}c_{jk}\left[(r-r_j)_{+,\infty}^{-k}+(-1)^k(r-r_j)_{-,r_j}^{-k}\right]\otimes1_{S^{n-1}}(r^{n-1}\phi/g)\nonumber\\&-&\sum_{j=1}^N\sum_{k=n}^{q_j}c_{jk}\frac{(-1)^{k-n}\ln r_j}{(k-n)!}\delta_0^{(k-n)}\otimes1_{S^{n-1}}(\phi/g)
\end{eqnarray}
for $\phi\in\mathcal{S}(\mathbf{R}^n)$. Here $\delta_0\in\mathcal{E}'(\mathbf{R})$ is Dirac's delta distribution. (The somewhat analogous distribution $\underline{x}^{-k}=x_+^{-k}+(-1)^kx_-^{-k}$ is defined in~\cite[p. 72]{HI} as the average of the distributions $(x+i0)^{-k}$ and $(x-i0)^{-k}$.)

Write $Q=\max\{q_j,\,\,j=1,\dots,N\}$. We have $Q\ge1$, as well as
\begin{lemma}\label{lemma:order}
$\mathfrak{p}^{-1}\in\mathcal{S}'^Q(\mathbf{R}^n)$.
\end{lemma}
\begin{proof}
For any $k\in\mathbf{N}$ and $b\in(0,\infty]$ there are constants $c$, $c'$ and $c''$ such that, for all $\phi\in\mathcal{S}(\mathbf{R}^n)$,
\begin{align*}
|(r-r_j)^{-k}_{\pm,b}\otimes1_{S^{n-1}}(r^{n-1}\phi/g)|&=\left|r^{-k}_{\pm,b}\otimes1_{S^{n-1}}\left(\frac{(r+r_j)^{n-1}\phi((r+r_j)\omega)}{g((r+r_j)\omega)}\right)\right|\\&\le\int_{\omega\in S^{n-1}}\Biggl\{c\int_{r=0}^b|\ln r|\left|\partial^k_r\left(\frac{(r+r_j)^{n-1}\phi((r+r_j)\omega)}{g((r+r_j)\omega)}\right)\right|(\pm r\omega)\Biggr.\\&\Biggl.+c'\left|\partial_r^{k-1}\left(\frac{(r+r_j)^{n-1}\phi((r+r_j)\omega)}{g((r+r_j)\omega)}\right)(0)\right|\Biggr.\\&\Biggl.+c''\sum_{\ell=0}^{k-2}\left|\partial_r^{\ell}\left(\frac{(r+r_j)^{n-1}\phi((r+r_j)\omega)}{g((r+r_j)\omega)}\right)(\pm b\omega)\right|\Biggr\}.
\end{align*}
Using Leibnitz' rule, we can express the $r$-derivatives in the above integrand in terms of $\partial^{\alpha}\phi$ with $|\alpha|\le k$. Finally, we have
\begin{align*}
\int_{\omega\in S^{n-1}}\int_{r=0}^b|\ln r|\cdot|\partial^{\alpha}\phi((\pm r+r_j)\omega)|&\le(\sup_{\xi\in\mathbf{R}^n}\langle\xi\rangle^k|\partial^{\alpha}\phi(\xi)|)\int_{r=0}^{\infty}|\ln r|\langle r\rangle^{-k}\\&\le c'''\sup_{\xi\in\mathbf{R}^n}\langle\xi\rangle^k|\partial^{\alpha}\phi(\xi)|
\end{align*}
for any multiindex $\alpha$ with $|\alpha|\le k$, where $c'''$ is some finite constant.
\end{proof}
In fact, $\mathfrak{p}^{-1}$ is the Fourier transform of a tempered fundamental solution of $P$:
\begin{lemma}\label{lemma:osobina} $p\cdot\mathfrak{p}^{-1}=1$ in $\mathcal{S}'(\mathbf{R}^n)$.
\end{lemma}
\begin{proof}
Fix $j\in\{1,\dots,N\}$, $k\in\{1,\dots,q_j\}$ and $\phi\in\mathcal{S}(\mathbf{R}^n)$. We have
\begin{gather*}
\left[(r-r_j)_+^{-k}+(-1)^k(r-r_j)_{-,r_j}^{-k}\right]\otimes1_{S^{n-1}}\left(\frac{r^{n-1}p(r\omega)\phi(r\omega)}{g(r\omega)}\right)\\=\left[r_+^{-k}+(-1)^kr_{-,r_j}^{-k}\right]\otimes1_{S^{n-1}}\left((r+r_j)^{n-1}\prod_{\ell=1}^N(r+r_j-r_{\ell})^{q_{\ell}}\phi((r+r_j)\omega)\right)\\=-\frac{1}{(k-1)!}\int_{r=-r_j}^{\infty}(\ln|r|)\partial_r^k\left((r+r_j)^{n-1}\prod_{\ell=1}^N(r+r_j-r_{\ell})^{q_{\ell}}1_{S^{n-1}}\phi((r+r_j)\omega)\right)\hspace{-1mm}(r)\\+\frac{1}{(k-1)!}\sum_{\nu=0}^{k-2}(-r_j)^{-k+\nu+1}(k-\nu-2)!\partial_r^{\nu}\left((r+r_j)^{n-1}\prod_{\ell=1}^N(r+r_j-r_{\ell})^{q_{\ell}}1_{S^{n-1}}\phi((r+r_j)\omega)\right)\hspace{-1mm}(-r_j).
\end{gather*}
Repeated integration by parts, well-defined because $r\mapsto 1_{S^{n-1}}(\phi(r\omega))$ and all its $r$-derivatives decay superpolynomially for increasing $r$, then gives
\begin{align*}
\mathfrak{p}^{-1}(\phi)&=\int_{r=0}^{\infty}(r-r_j)^{q_j-k}r^{n-1}\prod_{\substack{\ell=1,\dots,N\\\ell\neq j}}(r-r_{\ell})^{q_{\ell}}1_{S^{n-1}}\phi(r\omega)\\&+\frac{(-1)^{k-n}}{(k-n)!}\delta^{\rm Kr}_{k\ge n}(\ln r_j)\delta^{(k-n)}_{-r_j}\left(\prod_{\ell=1}^N(r+r_j-r_{\ell})^{q_{\ell}}1_{S^{n-1}}\phi((r+r_j)\omega)\right).
\end{align*}
We note that, for $k\ge n$,
\begin{gather*}
\partial_r^{k-1}\left((r+r_j)^{n-1}\prod_{\ell=1}^N(r+r_j-r_{\ell})^{q_{\ell}}1_{S^{n-1}}\phi((r+r_j)\omega)\right)\hspace{-1mm}(-r_j)=(k-n+1)\cdots(k-1)\\\times\left(\partial_r^{k-n}\prod_{\ell=1}^N(r+r_j-r_{\ell})^{q_{\ell}}1_{S^{n-1}}\phi((r+r_j)\omega)\right)\hspace{-1mm}(-r_j),
\end{gather*}
as well as that, for $2\le\mu\le k+1-n$,
\begin{gather*}
\partial_r^{k-\mu}\left((r+r_j)^{n-1}\prod_{\ell=1}^N(r+r_j-r_{\ell})^{q_{\ell}}1_{S^{n-1}}\phi((r+r_j)\omega)\right)\hspace{-1mm}(-r_j)=(k-\mu-n+2)\cdots(k-\mu)\\\times\left(\partial_r^{k-\mu-n+1}\prod_{\ell=1}^N(r+r_j-r_{\ell})^{q_{\ell}}1_{S^{n-1}}\phi((r+r_j)\omega)\right)\hspace{-1mm}(-r_j).
\end{gather*}
Finally, the symbol $p$ is smooth by definition, so if $\phi\in\mathcal{S}(\mathbf{R}^n)$ then
\begin{align*}
p\cdot\mathfrak{p}^{-1}(\phi)&=\mathfrak{p}^{-1}(p\phi)\\&=\sum_{j,k}c_{jk}\int_{r=-r_j}^{\infty}r^{q_j-k}(r+r_j)^{n-1}\times\prod_{\substack{\ell=1,\dots,N\\\ell\neq j}}(r+r_j-r_{\ell})^{q_{\ell}}1_{S^{n-1}}\phi((r+r_j)\omega)\\&=\sum_{j,k}c_{jk}\int_{r=0}^{\infty}(r-r_j)^{-k}r^{n-1}\prod_{\ell=1}^N(r-r_{\ell})^{q_{\ell}}1_{S^{n-1}}\phi(r\omega)\\&=\int_{r=0}^{\infty}r^{n-1}\sum_{j,k}c_{jk}(r-r_j)^{-k}\prod_{\ell=1}^N(r-r_{\ell})^{q_{\ell}}1_{S^{n-1}}\phi(r,\omega)=\int_{\mathbf{R}^n}\phi=1(\phi).
\end{align*}
\end{proof}

\section{Proof of Theorem~\ref{thm:MAIN}}\label{sec:proof}
\begin{lemma}\label{lemma:particul}
$\mathcal{F}^{-1}(\widehat{f}\mathfrak{p}^{-1})$ is a particular solution of $Pu=f$ in $\mathcal{S}'(\mathbf{R}^n)$.
\end{lemma}
\begin{proof}
Applying the Fourier transform to $Pu=f$ we get the equivalent expression $p\widehat{u}=\widehat{f}$, still in the sense of tempered distributions. Since $f$ is a compactly supported distribution, it is of some finite positive integer order $\ell$. By Theorem 7.1.14 in~\cite{HI}, we have $\widehat{f}\in C^{\infty}(\mathbf{R}^n)$, and by the Paley-Wiener-Schwartz theorem~\cite[Theorem 7.3.1 on p. 181]{HI} the function $\widehat{f}$ satisfies $|\widehat{f}(\xi)|\le C(1+|\xi|)^{\ell}$ for some constant $C$ and all $\xi\in\mathbf{R}^n$. Finally, by definition, $p\in C^{\infty}(\mathbf{R}^n)$ and $p\phi\in\mathcal{S}(\mathbf{R}^n)$ for every $\phi\in\mathcal{S}(\mathbf{R}^n)$, so
\[
\widehat{u}(p\phi)=p\widehat{u}(\phi)=\widehat{f}(\phi)=\int_{\mathbf{R}^n}\widehat{f}(\xi)\phi(\xi),\quad\phi\in\mathcal{S}(\mathbf{R}^n).
\]
Thus, any tempered distribution $v$ satisfying
\begin{equation}\label{eqn:sat}
v(p\phi)=\int_{\mathbf{R}^n}\phi(\xi),\quad\phi\in\mathcal{S}(\mathbf{R}^n),
\end{equation}
provides a solution $\widehat{u}=\widehat{f}v$ of $p\widehat{u}=\widehat{f}$ in $\mathcal{S}'(\mathbf{R}^n)$, since
\[
p\cdot(\widehat{f}v)(\phi)=v(p\widehat{f}\phi)=\int_{\mathbf{R}^n}\widehat{f}(\xi)\phi(\xi)=\widehat{f}(\phi),\quad\phi\in\mathcal{S}(\mathbf{R}^n).
\]
But we have shown in Lemma~\ref{lemma:osobina} that $\mathfrak{p}^{-1}$, defined in~\eqref{eqn:p_inverse}, satisfies~\eqref{eqn:sat}.
\end{proof}
\begin{lemma}\label{lemma:formula} If $u=\mathcal{F}^{-1}(\widehat{f}\mathfrak{p}^{-1})$ then, for every positive $\varrho$ and integer $m$,
\begin{align}\label{eqn:desired}
(-1)^n(2\pi)^{n-1}i^{-m}\widehat{U}^{\mathcal{C}}_{\varrho,m}&=\sum_{j=1}^N\sum_{k=1}^{q_j}c_{jk}\left[(r-r_j)_+^{-k}+(-1)^k(r-r_j)_{-,r_j}^{-k}\right]\otimes1_{S^{n-1}}\left(\frac{r^{n-1}\Psi_{\varrho,m}}{g}\right)\nonumber\\&-\sum_{j=1}^N\sum_{k=n}^{q_j}c_{jk}\frac{(-1)^{k-n}\ln r_j}{(k-n)!}\delta_0^{(k-n)}\otimes1_{S^{n-1}}(\Psi_{\varrho,m}/g),
\end{align}
where
\begin{equation}\label{eqn:Psi}
\Psi_{\varrho,m}(r\omega)=\widehat{f}(r\omega)\chi(r\omega/\varrho)J_m(Rr|\omega.\widetilde{e}|)e^{-im\angle(\omega.\widetilde{e})},\,\,r\ge0,\,\,\omega\in S^{n-1},
\end{equation}
and $\widetilde{e}=e_1+ie_2$.
\end{lemma}
\begin{proof}
In Section~\ref{section:introduction} we argued that $u$ is well-defined pointwise on $\mathcal{C}$ and $u_{\varrho}$ is well-defined pointwise on the whole $\mathbf{R}^n$. Therefore, writing $x(\theta)=e_1R\cos\theta+e_2R\sin\theta$ for $\theta\in[0,2\pi]$, using~\eqref{eqn:urho} and Lemma~\ref{lemma:integral}, as well as using the existence of $\widehat{U}^{\mathcal{C}}_{\varrho,m}$ and the linearity and the continuity of $\widehat{u}_{\varrho}$, we find that
\begin{eqnarray*}
\widehat{U}^{\mathcal{C}}_{\varrho,m}&=&\int_{\theta=0}^{2\pi}e^{-im\theta}u_{\varrho}(x(\theta))=(-2\pi)^{-n}\int_{\theta=0}^{2\pi}e^{-im\theta}(\widehat{u}_{\varrho})_{\xi}(e^{ix(\theta).\xi})\\&=&(-2\pi)^{-n}(\widehat{u}_{\varrho})_{\xi}\left(\int_{\theta=0}^{2\pi}e^{-im\theta}e^{ix(\theta).\xi}\right)\\&=&(-1)^n(2\pi)^{1-n}i^m\widehat{u}\left(\chi(\xi/\varrho)J_m(R|\xi.e_1+i\xi.e_2|)e^{-im\angle(e_1.\xi+ie_2.\xi)}\right)\\&=&(-1)^n(2\pi)^{1-n}i^m\mathfrak{p}^{-1}\left(\widehat{f}(\xi)\chi(\xi/\varrho)J_m(R|\xi.e_1+i\xi.e_2|)e^{-im\angle(e_1.\xi+ie_2.\xi)}\right)\hspace{-1mm}.
\end{eqnarray*}
This gives the desired expression~\eqref{eqn:desired} in light of~\eqref{eqn:p_inverse} and of the facts that
\[
|\xi.e_1+i\xi.e_2|=r|\omega.(e_1+ie_2)|\quad{\rm for}\,\,\xi\in\mathbf{R}^n
\]
and
\[
\angle(e_1.\xi+ie_2.\xi)=\angle r\omega.(e_1+ie_2)=\angle\omega.(e_1+ie_2)\,\,{\rm for}\,\,r>0.
\]
\end{proof}
\begin{remark}
It is a standard result that
\[
J_m(Rr)=\sqrt{\frac{2}{\pi Rr}}\cos(Rr-(2m+1)\pi/4)+O((Rr)^{-3/2}),\quad m\in\mathbf{Z},
\]
for $Rr\gg m^2$. Now for general $f\in\mathcal{E}'(\mathbf{R}^n)$ the function $\widehat{f}$ is not necessarily rapidly decaying, so neither is the argument
\[
\widehat{f}(\xi)J_m(R|\xi.e_1+i\xi.e_2|)\exp{-im\angle(e_1.\xi+ie_2.\xi)}
\]
of the tempered distribution $\mathfrak{p}^{-1}$ in the proof of Lemma~\ref{lemma:formula} if $\chi(\xi/\varrho)$ is omitted. This illustrates the need for the cut-off function $\chi(\cdot/\varrho)$ and for analyzing the approximate spectrum $\widehat{U}^{\mathcal{C}}_{\varrho,m}$, $\varrho>0$.
\end{remark}
\begin{corollary}\label{cor:asymp_particul}
If $u=\mathcal{F}^{-1}(\widehat{f}\mathfrak{p}^{-1})$ then there is a constant $c'$ and, for every positive $\varrho$, a constant $c_{\varrho}$ such that
\begin{align*}
\left|\widehat{U}^{\mathcal{C}}_{\varrho,m}\right|&\le c'+c_{\varrho}\sum_{j=1}^N\left(\max\{1,|m|^{q_j}\}|J_m(Rr_j)|+\max\{1,|m|^{q_j-1}\}|J_{m+1}(Rr_j)|\right),\quad m\in\mathbf{Z}.
\end{align*}
\end{corollary}
\begin{proof}
Fix $m\in\mathbf{Z}$, $j\in\{1,\dots,N\}$, and $k\in\{1,\dots,q_j\}$. A derivation similar to that in our proof of Lemma~\ref{lemma:osobina} shows that
\begin{eqnarray}\label{eqn:consequence}
(k-1)!\left[(r-r_j)_+^{-k}+(-1)^k(r-r_j)_{-,r_j}^{-k}\right]\otimes1_{S^{n-1}}\left(\frac{r^{n-1}\Psi_{\varrho,m}(r\omega)}{g(r\omega)}\right)\nonumber\\=-\int_{r=-r_j}^{\infty}(\ln|r|)\partial_r^k\left((r+r_j)^{n-1}\int_{\omega\in S^{n-1}}\frac{\Psi_{\varrho,m}((r+r_j)\omega)}{g((r+r_j)\omega)}\right)\hspace{-1mm}(r)\\+\sum_{\nu=0}^{k-2}(-r_j)^{-k+\nu+1}(k-\nu-2)!\partial_r^{\nu}\left((r+r_j)^{n-1}\int_{\omega\in S^{n-1}}\frac{\Psi_{\varrho,m}((r+r_j)\omega)}{g((r+r_j)\omega)}\right)\hspace{-1mm}(-r_j),\nonumber
\end{eqnarray}
where $\Psi_{\varrho,m}$ is defined in~\eqref{eqn:Psi}. The absolute value of the last sum in~\eqref{eqn:consequence}, and of the last double sum in~\eqref{eqn:desired}, are bounded by
\[
c'(r_j,k)\sum_{\nu=0}^{q_j-2}|\partial_r^{\nu}J_m(0)|,
\]
which via Lemma~\ref{lemma:diff_Bessel} leads to a bound of the form
\[
c'(r_j,k)|J_0(0)|+c''(r_j,k)\sum_{\nu=1}^{q_j-2}\left(|m|^{\nu}|J_m(0)|+|m|^{\nu-1}|J_{m+1}(0)|\right),
\]
that is, of the form $c'(r_j,k)+c''(r_j,k)$ when we recall that $J_0(0)=1$ and that $J_m(0)=0$ for any nonzero integer $m$. In light of~\eqref{eqn:desired}, it now remains to estimate the integral
\begin{equation}\label{eqn:integ_important}
\int_{r=-r_j}^{\infty}(\ln|r|)\partial_r^k\left((r+r_j)^{n-1}\int_{\omega\in S^{n-1}}\frac{\Psi_{\varrho,m}((r+r_j)\omega)}{g((r+r_j)\omega)}\right)\hspace{-1mm}(r)
\end{equation}
occurring in~\eqref{eqn:consequence}. We split the integral into three parts, and estimate each part separately. Fix $\omega\in S^{n-1}$ and $\varepsilon\in(0,\min\{1,r_j\})$. To simplify the notation, we write $\phi(r)=r^{n-1}1_{S^{n-1}}\Psi_{\varrho,m}(r\omega)/g(r\omega)$ for $r\ge0$. In particular, $\phi\in C^{\infty}(0,\infty)$ and $\partial_r^{\ell}\phi(r)=0$ for $r>2\varrho$ and for all $\ell\in\mathbf{N}_0$. First, if
\[
C'_k\ge1+\frac{\max_{-\varepsilon\le t<\infty}|t^{k+1}\phi^{(k+1)}(t+r_j)|}{(k+1)!|\phi^{(k)}(r_j)|}
\]
then $|\phi^{(k)}(r+r_j)|\le C'_k|\phi^{(k)}(r_j)|$ and consequently
\begin{align*}
\left|\int_{r=-\varepsilon}^{\varepsilon}\ln(|r|)\phi^{(k)}(r+r_j)\right|\le-2C'_k\left|\phi^{(k)}(r_j)\right|\int_{r=0}^{\varepsilon}\ln r=2C'_k\varepsilon(1+|\ln\varepsilon|)\left|\phi^{(k)}(r_j)\right|.
\end{align*}
Next, repeated integration by parts gives
\begin{align*}
\left|\int_{r=\varepsilon}^\infty(\ln r)\phi^{(k)}(r+r_j)\right|&=\left|\int_{r=\varepsilon}^{2\varrho}(\ln r)\phi^{(k)}(r+r_j)\right|\\&=\Biggl|-\ln(\varepsilon)\phi^{(k-1)}(r_j+\varepsilon)+\sum_{\ell=1}^{k-1}(\ell-1)!\phi^{(k-\ell-1)}(r_j+\varepsilon)\varepsilon^{-\ell}\\&-(k-1)!\int_{r=\varepsilon}^{2\varrho}\phi(r+r_j)r^{-k}\Biggr|\\&\le|\ln\varepsilon|\cdot\left|\phi^{(k-1)}(r_j+\varepsilon)\right|+\sum_{\ell=0}^{k-2}(k-\ell-2)!\varepsilon^{\ell-k+1}\left|\phi^{(\ell)}(r_j+\varepsilon)\right|\\&+\delta^{\rm Kr}_{k=1}C_0'\left|\ln(2\varrho/\varepsilon)\right|\cdot\left|\phi(r_j)\right|+\delta^{\rm Kr}_{k\ge2}C_0'(k-2)!|(2\varrho)^{1-k}-\varepsilon^{1-k}|\cdot\left|\phi(r_j)\right|.
\end{align*}
Finally, repeated integration by parts gives
\begin{align*}
\left|\int_{r=-r_j}^{-\varepsilon}\ln(|r|)\phi^{(k)}(r+r_j)\right|&=\Biggl|-\ln(r_j)\phi^{(k-1)}(0)+\sum_{\ell=2}^k(-1)^{\ell+1}r_j^{1-\ell}\phi^{(k-\ell)}(0)(\ell-2)!\\&+\ln(\varepsilon)\phi^{(k-1)}(r_j-\varepsilon)+\sum_{\ell=2}^k(-1)^{\ell}(\ell-2)!\varepsilon^{1-\ell}\phi^{(k-\ell)}(r_j-\varepsilon)\\&+(-1)^{k+1}(k-1)!\int_{r=\varepsilon}^{r_j}r^{-k}\phi(r_j-r)\Biggr|\\&\le c''(r_j,\varepsilon,k)\sum_{\ell=0}^{k-1}C'_{\ell}\left|\phi^{(\ell)}(r_j)\right|+\delta^{\rm Kr}_{k=1}C_0'\left|\ln(r_j/\varepsilon)\right|\cdot\left|\phi(r_j)\right|\\&+\delta^{\rm Kr}_{k\ge2}C_0'(k-2)!\left|r_j^{1-k}-\varepsilon^{1-k}\right|\cdot\left|\phi(r_j)\right|.
\end{align*}
In conclusion, the absolute value of the integral in~\eqref{eqn:integ_important} is bounded from above by a sum of the form
\[
c'(\varrho,r_j,\varepsilon,k)\sum_{\nu=0}^{q_j}|\partial_r^{\nu}J_m(Rr_j)|,
\]
which via Lemma~\ref{lemma:diff_Bessel} leads to an estimate of the form
\[
c''(\varrho,r_j,\varepsilon,k)\left[|J_m(Rr_j)|+\sum_{\nu=1}^{q_j}\left(|m|^{\nu}|J_m(Rr_j)|+|m|^{\nu-1}|J_{m+1}(Rr_j)|\right)\right],
\]
hence to estimates of the form
\[
c'''(\varrho,r_j,\varepsilon,k)\left[J_0(Rr_j)|+|J_1(Rr_j)|\right],\quad{\rm for}\,\,\,m=0,
\]
and
\[
c'''(\varrho,r_j,\varepsilon,k)\left(|m|^{q_j}|J_m(Rr_j)|+|m|^{q_j-1}|J_{m+1}(Rr_j)|\right),\quad {\rm for}\,\,\,m\neq0.
\]
\end{proof}

It remains to characterize the homogeneous solutions.
Assume $u\in\mathcal{S}'^d(\mathbf{R}^n)$ satisfies $Pu=0$ in $\mathcal{S}'(\mathbf{R}^n)$.
\begin{lemma}\label{lemma:PWd}
$\widehat{u}\in\mathcal{E}'^d(\mathbf{R}^n)$.
\end{lemma}
\begin{proof}
If $\phi\in C_0^{\infty}(\mathbf{R}^n)$ has its support in the complement of the null-set $p^{-1}(\{0\})$ of $p$ then $\phi/p\in C_0^{\infty}(\mathbf{R}^n)$ and $\widehat{u}(\phi)=p\widehat{u}(\phi/p)=0$, so $\supp\widehat{u}\subseteq p^{-1}(\{0\})$ and in particular $\widehat{u}\in\mathcal{E}'(\mathbf{R}^n)$. To find the order of $\widehat{u}$, first note that
\begin{align*}
    \langle\xi\rangle^d
    \le
    (1+d\max_{j} |\xi_j|^2 )^\frac{d}{2}
    &\le
    d^{\frac{d}{2}}(1+\max_{j} |\xi_j|)^d \\
    &\le d^{\frac{d}{2}}\Big(1+\sum_{j}|\xi_j|\Big)^d 
    \le
    d^{\frac{d}{2}} (d+1)! \sum_{|\beta|\leq d} |\xi^\beta|,
\end{align*}
and let $\psi\in C_0^{\infty}(\mathbf{R}^n)$ satisfy $\psi(\xi)=1$ for $|\xi|\le\max_j\{r_j\}$. Then, for any $\phi\in C^{\infty}(\mathbf{R}^n)$, we obtain the estimates
\begin{align*}
|\widehat{u}(\phi)|&=|\widehat{u}(\psi\phi)|=|u(\widehat{\psi\phi})|\le C\sum_{|\alpha|\le d}\sup_{\xi\in\mathbf{R}^n}\langle\xi\rangle^d|\partial^{\alpha}\widehat{\psi\phi}(\xi)|\\
&\le C'\sum_{|\alpha|\le d}\sup_{\xi\in\mathbf{R}^n} \Big[ \sum_{|\beta|\le d} |\xi^\beta| \Big] |\mathcal{F}(x^{\alpha}\psi\phi)(\xi)|\\
&\le C'' \sum_{|\alpha|\le d} \sum_{|\beta|\le d} \int_{x\in\supp\psi}|\partial^\beta_x [x^{\alpha}(\psi\phi)(x)]| \\
&\le C'''\sup_{x\in\supp\psi}\sum_{|\beta|\le d}|\partial^{\beta}\phi(x)|,
\end{align*}
where the last inequality follows after repeated application of Leibnitz' rule, and the constants $C'$, $C''$ and $C'''$ are all $\phi$-independent. 
\end{proof}
Define $\Phi : (0,\infty) \times S^{n-1} \to \mathbf{R}^n\setminus \{0\}$ by $\Phi(r,\omega) = r\omega$.
\begin{lemma} 
There are $u_{k,j} \in \mathcal{D}'^{d-k}(S^{n-1})$ such that
\begin{equation}\label{eqn:one}
u(x)
=  
\sum^N_{j=1} \sum^d_{k=0} \left(\partial^k_r \delta_{r_j}(r) \otimes u_{k,j}(\omega)\right)\left( e^{ir x.\omega} \right),
\quad
x \in \mathbf{C}^n,
\end{equation}
where
\begin{align}\label{eqn:two} 
\sum_{k=q_j}^d  \binom{k}{q_j} (-1)^k  u_{k,j} = 0,
\quad
j\in \{1, \cdots, N\}.
\end{align}
\end{lemma}
\begin{proof}

Since $\widehat{u}\in\mathcal{E}'^d(\mathbf{R}^n)$, we can use cutoffs to arrive at a decomposition $\widehat{u} = \sum_{j=1}^N v_j$ for some distributions $v_j\in\mathcal{E}'^d(\mathbf{R}^n)$ with $\supp(v_j) \subset \Phi(\{r_j\}\times S^{n-1})$. Then, for every $\phi \in C^\infty_0(\mathbf{R}^n \setminus\{0\})$ we have
\begin{align*}
0 = pv_j(\phi)=(\Phi^*v_j)\left(\left(g\prod_{k=1}^N(r-r_k)^{q_k}\right)(\phi\circ \Phi)|\det d\Phi |\right),
\end{align*}
and this leads to
\begin{align*}
(\Phi^*v_j(r,\omega))((r-r_j)^{q_j} \phi(r,\omega)) = 0,
\quad
\phi \in C^\infty_0( (0,\infty) \times S^{n-1} ).
\end{align*}
By~\cite[Theorem~2.3.5]{HI} or~\cite[Example~5.1.2]{HI},
\begin{align*}
\Phi^*v_j = \sum_{k=0}^d \partial^k_r \delta_{r_j} \otimes u_{k,j},
\end{align*}
where the distributions $u_{k,j} \in \mathcal{D}'^{d-k}(S^{n-1})$ must satisfy
\begin{align*}
0&= \sum_{k=0}^d \partial^k \delta_{r_j}(r) \otimes u_{k,j}(\omega)\left((r-r_j)^{q_j} \phi(r,\omega)\right) 
\\&=
\sum_{k=q_j}^d  \binom{k}{q_j} (-1)^k u_{k,j}(\omega)(\partial^{k-q_j}_r\phi(r_j,\omega)), \quad\phi \in C^\infty_0( (0,\infty) \times S^{n-1} ).
\end{align*}
\end{proof}
In the following, we assume without loss of generality that the radius $R$ of the great circle $\mathcal{C}$ is 1. Inserting~\eqref{eqn:one} into~\eqref{eqn:FC}, we get
\begin{align*}
    \widehat{U}^{\mathcal{C}}_m
    &=
    \sum^N_{j=1} \sum^d_{k=0} u_{k,j}(\omega)\left((-1)^k \lim_{r \to r_j} \partial_r^k \int_{0}^{2\pi} e^{ir ( \omega_1 \cos(\theta) + \omega_2 \sin(\theta)) -im\theta} \, d\theta \right) 
    \\&= 
    \sum^N_{j=1} \sum^d_{k=0} u_{k,j}(\omega)\left([g_{k,j}(m)](\omega)  \right),
\end{align*}
where
\begin{align*}
    g_{k,j}(m)(\omega) = 2\pi i^m (-1)^k |\omega_1+i \omega_2|^k \partial_r^k  J_m(r_j |\omega_1+i \omega_2|)e^{-i m\arg (\omega_1+i\omega_2)}.
\end{align*}
Note that each $g_{k,j}$ is smooth, so the above action of $u_{k,j}$ on $g_{k,j}$ is well-defined. In fact, each $g_{k,j}$ is in hyperspherical coordinates given by
\begin{align*}
    g_{k,j}(m)(\omega) 
    &= 
    2\pi i^m (\omega_1 - i\omega_2)^m 
    \sum_{2s+m\geq k}^\infty \frac{(-1)^{s+k} (2s+m)_k}{s!(s+m)!2^{2s+m}}
    r^{2s+m-k}_j (\omega_1^2 + \omega_2^2)^{s} \\
    &=
    2\pi i^m(-1)^k e^{-im\theta_{n-1}} \Big(\prod_{s=2}^{n-1} \sin(\theta_{n-s}) \Big)^{k} \partial^k_r J_m\Big(r_j \prod_{s=2}^{n-1} \sin(\theta_{n-s})  \Big),
\end{align*}
where $\theta_1, \cdots, \theta_{n-2}$ run in $[0,\pi)$, $\theta_{n-1}$ in $[0,2\pi)$, and
\begin{align*}
\omega_1 = \sin(\theta_1) \cdots \sin(\theta_{n-2})\sin(\theta_{n-1}), \\
\omega_2 = \sin(\theta_1) \cdots \sin(\theta_{n-2})\cos(\theta_{n-1}).
\end{align*}
Since $u_{k,j}$ is of order $d-k$ there are $A_{k,j} \in \textnormal{Diff}^{d-k}(S^{n-1})$ such that 
\begin{align*}
    \left|\widehat{U}^{\mathcal{C}}_m\right|
    \le C\Big(\sum^N_{j=1} \sum^d_{k=0}  \sup_{\omega\in S^{n-1}} |  A_{k,j} [g_{k,j}(m)](\omega) | \Big)
\end{align*}
for some constant $C$. Using Leibniz' rule and Lemma~\ref{lemma:diff_Bessel}, we get
\begin{align*}
    \left|\widehat{U}^{\mathcal{C}}_m\right|
    &\le C'\sum^N_{j=1} \sum^d_{k=0} \sum_{s=0}^{d-k} \sum_{s'=0}^{d-k-s} |m|^s|\partial_r^{k+s'}J_m(r_j)|
    =C'\sum_{j=1}^N\sum_{k=0}^d\sum_{s=0}^{d-k}\sum_{s''=k}^{d-s}|m|^s|\partial_r^{s''}J_m(r_j)|
    \\&\le \widetilde{C}\sum_{j=1}^N\sum_{s=0}^d|m|^s|J_m(r_j)|+C''\sum^N_{j=1} \sum_{\substack{k,s,s''\\s''\ge1}}\left(|m|^{s+s''}|J_m(r_j)|+|m|^{s+s''-1}|J_{m+1}(r_j)|\right)\\&\le C'''\sum_{j=1}^N\left(\max\{1,|m|^d|\}|J_m(r_j)|+\delta^{\rm Kr}_{d\ge1}\max\{1,|m|^{d-1}\}|J_{m+1}(r_j)|\right),\quad m\in\mathbf{Z},
\end{align*}
for some constant $C'$, $C''$, $C'''$ and $\widetilde{C}$. We have here used that $s+s''\le d$ since the index $s''$ ranges up to $d-s$. Returning to general values of the radius $R$, we have thus shown
\begin{theorem}\label{thm:homogen}
If $u \in \mathcal{S}'^d(\mathbf{R}^n)$ solves $Pu = 0$ in $\mathcal{S}'(\mathbf{R}^n)$ then there is a constant $c$ satisfying
\begin{align*}
    \left|\widehat{U}^{\mathcal{C}}_m\right|\le c\sum_{j=1}^N\left(\max\{1,|m|^d\}|J_m(r_jR)|+\delta^{\rm Kr}_{d\ge1}\max\{1,|m|^{d-1}\}|J_{m+1}(r_jR)|\right) ,\quad m\in\mathbf{Z}.
\end{align*}
\end{theorem}

\section{Application of Theorem~\ref{thm:MAIN} to inverse source problems}\label{sec:ispHelm}

We here study the simple but important special case $P=\Delta+k^2$ in $\mathbf{R}^n$, where the 'wavenumber' $k$ is a nonzero complex constant with $\Im k\ge0$. We test numerically the upper bound of Theorem~\ref{thm:MAIN} on the position of the spectral cutoff in the measurement $u|_{\mathcal{C}}$. This spectral cutoff is directly relevant for the stability of the inverse source problem for the Helmholtz equation in $\mathbf{R}^n$.~\cite{2018a-bandwidth,2020a-Helmholtz_3D}

\subsection{Point source}\label{subsec:pointsource}

Consider first the inhomogeneous Helmholtz equation
\begin{equation}\label{eqn:Helm}
(\Delta+k^2)u=\partial_j^{\nu}\delta_y\qquad\text{in}\,\,\,\textbf{R}^n,
\end{equation}
where the inhomogeneity is a 'point source located at $y\in\mathbf{R}^n$.' Here $j\in\{1,\dots,n\}$ and $\nu\in\mathbf{N}_0$ are fixed. To ensure uniqueness of solution of~\eqref{eqn:Helm}, we impose the appropriate Sommerfeld radiation condition~\cite[Eq. (7.2)]{Mitrea-1996}. The unique outgoing fundamental solution of the Helmholtz operator is~\cite[Eq. (16)]{McIntoshMitrea-1999}
\[
\Phi_n(x)=\begin{cases}(-2\pi|x|)^{(1-n)/2}(2ik)^{-1}\partial_{|x|}^{(n-1)/2}e^{ik|x|},\quad&x\in\mathbf{R}^n\setminus\{0\},\,\,\,n\,\,\,\text{odd},\\(-2\pi|x|)^{(2-n)/2}(4i)^{-1}\partial_{|x|}^{(n-2)/2}H_0^{(1)}(k|x|),\quad&x\in\mathbf{R}^n\setminus\{0\},\,\,\,n\,\,\,\,\text{even},\end{cases}
\]
with $H_0^{(1)}$ the Hankel function of the first kind and order zero. In particular, $\Phi_n\in C^{\infty}(\mathbf{R}^n\setminus\{0\})$. To apply Theorem~\ref{thm:MAIN} to~\eqref{eqn:Helm}, we need the order of $u=(\partial_j^{\nu}\delta_y)\ast\Phi_n=(-1)^{\nu}(\partial_j^{\nu}\Phi_n)(\cdot-y)$ as a tempered distribution.
\begin{lemma}\label{lemma:H}
$\Phi_n\in\mathcal{S}'^{(n+3)/2}(\mathbf{R}^n)$ for $n$ odd. Furthermore, $\Phi_2\in\mathcal{S}'^2(\mathbf{R}^2)$, $\Phi_4\in\mathcal{S}'^3(\mathbf{R}^4)$ and $\Phi_n\in\mathcal{S}'^4(\mathbf{R}^n)$ for $n\in\{6,8,10,\dots\}$.
\end{lemma}
\begin{proof}
With $n$ odd and $\phi\in\mathcal{S}(\mathbf{R}^n)$, we have
\begin{align*}
|\Phi_n(\phi)|&=c_n\left|\int_{r=0}^{\infty}\int_{\omega\in S^{n-1}}r^{(n-1)/2}\phi(r\omega)\partial_r^{(n-1)/2}e^{ikr}\right|\\&\le c'_n\int_{r=0}^{\infty}r^{(n-1)/2}\langle r\rangle^{-N}\int_{\omega\in S^{n-1}}\langle r\rangle^N|\phi(r\omega)|\\&\le c''_n\left(\sup_{x\in\mathbf{R}^n}\langle x\rangle^N|\phi(x)|\right)\int_{r=0}^{\infty}r^{(n-1)/2}\langle r\rangle^{-N}\\&\le c'''_{n,N}\sum_{|\alpha|\le N}\sup_{x\in\mathbf{R}^n}\langle x\rangle^N|\partial^{\alpha}\phi(x)|,
\end{align*}
where $c'''_{n,N}$ is a finite constant if $N>(n+1)/2$. Next, assume
\[
\partial^q_rH_0^{(1)}(kr)=\pi_{q-2}(1/r)H_0^{(1)}(kr)+\pi_{q-1}(1/r)H_1^{(1)}(kr)
\]
for some integer $q\ge2$, and where $\pi_w$ signifies 'some polynomial of degree $w$.' Then
\begin{align*}
\partial_r^{q+1}H_0^{(1)}(kr)&=\pi_{q-1}(1/r)H_0^{(1)}(kr)+\pi_{q-2}(1/r)(-H_1^{(1)}(kr))\\&+\pi_q(1/r)H_1^{(1)}(kr)+\pi_{q-1}(1/r)(H_0^{(1)}(kr)-\frac{1}{r}H_1^{(1)}(kr))\\&=\pi_{q-1}(1/r)H_0^{(1)}(kr)+\pi_q(1/r)H_1^{(1)}(kr),
\end{align*}
and finally
\[
\partial^2_rH_0^{(1)}(kr)=H_0^{(1)}(kr)-(1/r)H_1^{(1)}(kr)=\pi_0(1/r)H_0^{(1)}(kr)+\pi_1(1/r)H_1^{(1)}(kr).
\]
Now fix $n$ even and $\phi\in\mathcal{S}(\mathbf{R}^n)$. We have
\begin{align*}
|\Phi_n(\phi)|&=c_n\left|\int_{r=0}^{\infty}\int_{\omega\in S^{n-1}}r^{n/2}\phi(r\omega)\partial_r^{n/2-1}H_0^{(1)}(kr)\right|\\&\le c_n'\int_{r=0}^{\infty}r^{n/2}\langle r\rangle^{-N}|\partial_r^{n/2-1}H_0^{(1)}(kr)|\int_{\omega\in S^{n-1}}\langle r\rangle^N|\phi(r\omega)|\\&\le c''_n\left(\sup_{x\in\mathbf{R}^n}\langle x\rangle^N|\phi(x)|\right)\int_{r=0}^{\infty}r^{n/2}\langle r\rangle^{-N}|\partial_r^{n/2-1}H_0^{(1)}(kr)|\\&\le c_n''\left(\sum_{|\alpha|\le N}\sup_{x\in\mathbf{R}^n}\langle x\rangle^N|\partial^{\alpha}\phi(x)|\right)\int_{r=0}^{\infty}r^{n/2}\langle r\rangle^{-N}|\partial_r^{n/2-1}H_0^{(1)}(kr)|.
\end{align*}
The functions $H_0^{(1)}(kr)$ and $H_1^{(1)}(kr)$ have singularities precisely at $r=0$, and these singularities are $\ln(r)$-like and $r^{-1}$-like, respectively. Therefore, if $n=2$, we have
\begin{align*}
\int_{r=0}^{\infty}r^{n/2}\langle r\rangle^{-N}|\partial_r^{n/2-1}H_0^{(1)}(kr)|&=\int_{r=0}^{\infty}\frac{r}{\langle r\rangle^N}|H_0^{(1)}(kr)|\\&\le\int_{r=0}^1\frac{r}{\langle r\rangle^N}|H_0^{(1)}(kr)|+c\int_{r\ge1}\frac{r}{\langle r\rangle^N}r^{-1/2},
\end{align*}
which is finite if $N\ge2$. If $n=4$, we have
\begin{align*}
\int_{r=0}^{\infty}r^{n/2}\langle r\rangle^{-N}|\partial_r^{n/2-1}H_0^{(1)}(kr)|&=\int_{r=0}^{\infty}\frac{r^2}{\langle r\rangle^N}|\partial_rH_0^{(1)}(kr)|=\int_{r=0}^{\infty}\frac{r^2}{\langle r\rangle^N}|H_1^{(1)}(kr)|\\&\le\int_{r=0}^1\frac{r^2}{\langle r\rangle^N}|H_1^{(1)}(kr)|+c\int_{r\ge1}\frac{r^2}{\langle r\rangle^N}r^{-1/2},
\end{align*}
which is finite if $N\ge3$. Finally, if $n\in\{6,8,10,\dots\}$ then, by the inductive argument above,
\begin{align*}
\int_{r=0}^{\infty}r^{n/2}\langle r\rangle^{-N}|\partial_r^{n/2-1}H_0^{(1)}(kr)|&\le\int_{r=0}^{\infty}\frac{r^{n/2}}{\langle r\rangle^N}|\pi_{n/2-3}(1/r)H_0^{(1)}(kr)|\\&+\int_{r=0}^{\infty}\frac{r^{n/2}}{\langle r\rangle^N}|\pi_{n/2-2}(1/r)H_1^{(1)}(kr)|\\&\le c_n+c'\int_{r=1}^{\infty}\frac{r^{n/2}}{\langle r\rangle^N}r^{-n/2+3}r^{-1/2},
\end{align*}
which is finite if $N\ge4$.
\end{proof}
\begin{remark}
Our estimates of the distributional order in $\mathcal{S}'(\mathbf{R}^n)$ of radially outgoing fundamental solutions $\Phi_n$ coincide for $n=1$ and $n=2$; for $n=3$ and $n=4$; and for $n=5$ and $n=6,8,10,\dots$.
\end{remark}
Write $d(n)$ for the order of $\Phi_n$ as a tempered distribution, estimated in Lemma~\ref{lemma:H}.
\begin{corollary}\label{corr:order} $(\partial_j^{\nu}\Phi_n)(\cdot-y)\in\mathcal{S}'^{d(n)+\nu}(\mathbf{R}^n)$.
\end{corollary}
\begin{proof}
We readily get
\begin{align*}
|(\partial_j^{\nu}\Phi_n(\cdot-y))(\phi)|&=|\Phi_n((\partial_j^{\nu}\delta_y)(-\cdot)\ast\phi)|\\&\le C\sum_{|\alpha|\le d(n)}\langle x\rangle^{d(n)}\sup_{x\in\mathbf{R}^n}|(\partial_j^{\nu}\delta_y)\ast\partial^{\alpha}\phi(x)|\\&\le C\sum_{|\alpha|\le d(n)+\nu}\langle x\rangle^{d(n)+\nu}\sup_{x\in\mathbf{R}^n}|\partial^{\alpha}\phi(x)|
\end{align*}
for any $\phi\in\mathcal{S}(\mathbf{R}^n)$.
\end{proof}
Given $R>0$, $e_1,e_2\in S^{n-1}$, $e_1\perp e_2$, and the circle \[\mathcal{C}=R(e_1\cos\theta+e_2\sin\theta),\,\,\,\theta\in[0,2\pi],\] disjoint from the source support $\{y\}$, we have $U^{\mathcal{C}}=(-1)^{\nu}(\partial_j^{\nu}\Phi_n)(\cdot-y)|_{\mathcal{C}}$ and
\begin{equation}\label{eqn:direct}
\widehat{U}^{\mathcal{C}}_m=(-1)^{\nu}\int_{\theta=0}^{2\pi}e^{-im\theta}\partial_j^{\nu}\Phi_n(e_1R\cos\theta+e_2R\sin\theta-y),\quad m\in\mathbf{Z}.
\end{equation}
We can of course compute numerically concrete values $|\widehat{U}^{\mathcal{C}}_m|$ of the measurement spectrum from~\eqref{eqn:direct}. Next, we here recall for convenience that $\chi\in C_0^{\infty}(\mathbf{R}^n)$ satisfies
\[
\chi(\xi)=\begin{cases}1,\quad&|\xi|\le 1,\\0,\quad&|\xi|\ge2,\end{cases}
\]
and that $\mathcal{F}(\partial_j^{\nu}\Phi_n(\cdot-y)_{\varrho})=\chi(\cdot/\varrho)\mathcal{F}(\partial_j^{\nu}\Phi_n(\cdot-y))=\chi(\cdot/\varrho)e^{-iy\xi}\widehat{\partial_j^{\nu}\Phi}_n\in\mathcal{E}'(\mathbf{R}^n)$ for every positive $\varrho$. Also,
\[
\widehat{U}^{\mathcal{C}}_{\varrho,m}=(-1)^{\nu}\int_{\theta=0}^{2\pi}e^{-im\theta}(\partial_j^{\nu}\Phi_n)_{\varrho}(e_1R\cos\theta+e_2R\sin\theta-y),\quad m\in\mathbf{Z}.
\]
In the present special case, we have $p(\xi)=k^2-|\xi|^2=-(|\xi|+k)(|\xi|-k)$, $g(r\omega)=-(r+k)$, $N=1$, $r_1=k$, $q_1=1$, and $c_{11}=1$. In light of this fact, as well as Lemma~\ref{lemma:H} and Corollary~\ref{corr:order}, Theorem~\ref{thm:MAIN} guarantees that, for every positive $\varrho$,
\begin{equation}\label{eqn:guarantee}
\left|\widehat{U}^{\mathcal{C}}_{\varrho,m}\right|\le c'+c_{\varrho}\times\begin{cases}\left(|m|^{(n+3+2\nu)/2}\cdot|J_m(kR)|+|m|^{(n+1+2\nu)/2}|J_{m+1}(kR)|\right),\quad n\,\,\,\text{odd},\\\left(|m|^{2+\nu}\cdot|J_m(kR)|+|m|^{1+\nu}\cdot|J_{m+1}(kR)|\right),\quad n=2,\\\left(|m|^{3+\nu}\cdot|J_m(kR)|+|m|^{2+\nu}|J_{m+1}(kR)|\right),\quad\,\, n=4,\\\left(|m|^{4+\nu}\cdot|J_m(kR)|+|m|^{3+\nu}|J_{m+1}(kR)|\right),\quad\,\,n=6,8,10,\dots,\end{cases}
\end{equation}
where $c'$ and $c_{\varrho}$ are (unknown) constants. Note that the $m$-dependent part of~\eqref{eqn:guarantee} is the same for dimensions $n=j$ and $n=j+1$ with $j=1,3,5$, as well as for $n=5$ and $n=8,10,12,\dots$. Figure~\ref{fig:compar} shows the estimate from~\eqref{eqn:guarantee}, as well as the values $|\widehat{U}^{\mathcal{C}}_m|$ from~\eqref{eqn:direct} with $\nu=0$, $n=2,3,4,5$, $m=0,\dots,100$, $k=2\pi$, $R=5.01$, and $|y|=5$. We shift and scale the computed spectra~\eqref{eqn:direct} and the upper bound estimates~\eqref{eqn:guarantee} such that they range in the interval $[-1,1]$. We do this because the absolute level of the values from~\eqref{eqn:guarantee} is arbitrary, and we are only interested in the spectral location of the onset of rapid decay of $\widehat{U}^{\mathcal C}_m$. Figure~\ref{fig:compar} indicates that Theorem~\ref{thm:MAIN} predicts this onset well for $n=2,\dots,5$, and Table~\ref{tab:1} on page~\pageref{tab:1} lists the predicted and the actual 'bandwidths' for these values of $n$.
\begin{figure}[hbt]
\hspace{0cm}\includegraphics[scale=.6]{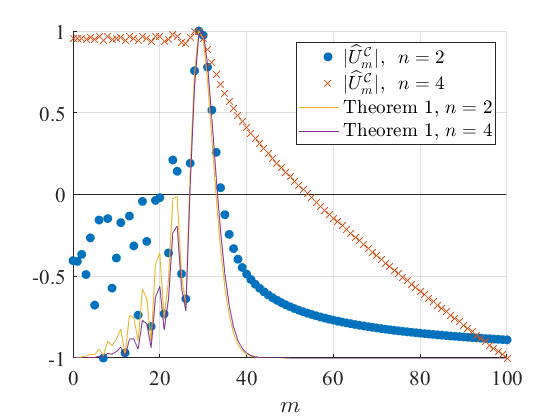}\includegraphics[scale=.6]{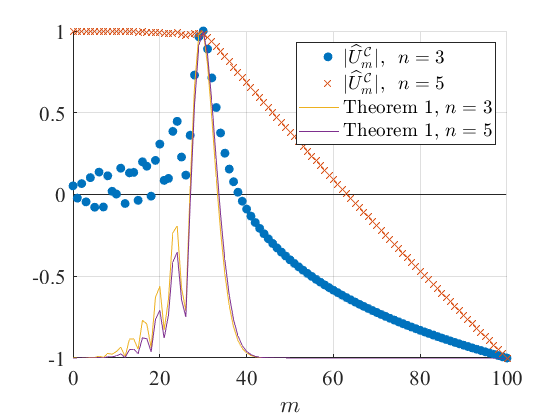}
\caption{Comparison of the actual spectrum $|\widehat{U}^{\mathcal C}_m|$ with the bound from Theorem~\ref{thm:MAIN}. The spectrum and the bound are shifted and scaled to range in $[-1,1]$. We are interested in the spectral location of the onset of rapid decay of $|\widehat{U}^{\mathcal{C}}_m|$. Parameters: $\nu=0$, $R=5.01$, $|y|=5$.}\label{fig:compar}
\end{figure}
\begin{table}[hbt]
\begin{center}
\begin{tabular}{ |c|c|c| } 
 \hline
 dimension $n$ & bandwidth predicted & actual \\
               & by Theorem~\ref{thm:MAIN} & bandwidth\\
 \hline
 2 & 29 & 29 \\ 
 3 & 30 & 30 \\ 
 4 & 30 & 29 \\
 5 & 30 & 29 \\
 \hline
\end{tabular}
\end{center}
\caption{Predicted and measured bandwidths of the spectrum $\widehat{U}^{\mathcal{C}}_m$ from Figure~\ref{fig:compar}. Parameters: $\nu=0$, $R=5.01$, $|y|=5$.}\label{tab:1}
\end{table}
Figure~\ref{fig:higherD} on page~\pageref{fig:higherD} shows the estimate from~\eqref{eqn:guarantee} and the values $|\widehat{U}^{\mathcal{C}}_m|$ from~\eqref{eqn:direct} with $\nu=0$, $n=6,7,8,9$, $m=0,\dots,100$, $k=2\pi$, $R=5$, and $y=(10,0,\cdots,0)$; hence in this case the 'measurement set' envelops the support of the 'source.'
\begin{figure}[hbt]
\hspace{0cm}\includegraphics[scale=0.6]{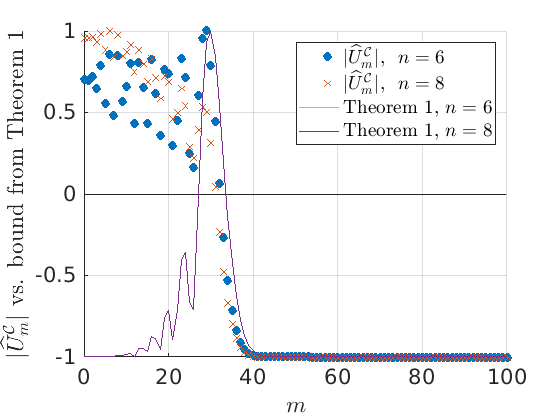}\includegraphics[scale=0.6]{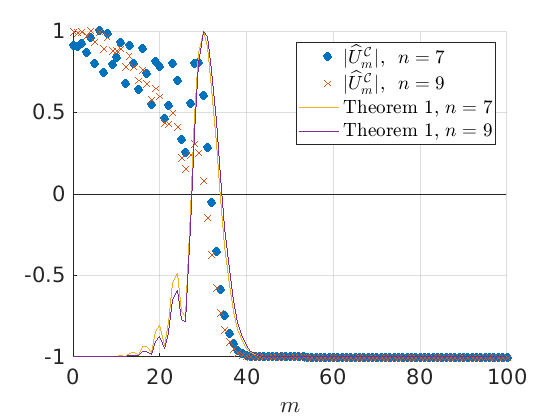}
\caption{Comparison of the actual spectrum $|\widehat{U}^{\mathcal C}_m|$ with the bound from Theorem~\ref{thm:MAIN}. The spectrum and the bound are shifted and scaled to range in $[-1,1]$. We are interested in the spectral location of the onset of rapid decay of $|\widehat{U}^{\mathcal{C}}_m|$. Parameters: $\nu=0$, $R=5$, $y=(10,0,\cdots,0)$.}\label{fig:higherD}
\end{figure}
Theorem~\ref{thm:MAIN} again lets us calculate tight upper bounds on the measurement bandwidth, as is evident in Table~\ref{tab:2}.
\begin{table}[hbt]
\begin{center}
\begin{tabular}{ |c|c|c| } 
 \hline
 dimension $n$ & bandwidth predicted & actual \\
               & by Theorem~\ref{thm:MAIN} & bandwidth\\
 \hline
 6 & 30 & 29 \\ 
 7 & 30 & 29 \\ 
 8 & 30 & 28 \\
 9 & 30 & 28 \\
 \hline
\end{tabular}
\end{center}
\caption{Predicted and measured bandwidths of the spectrum $\widehat{U}^{\mathcal{C}}_m$ from Figure~\ref{fig:higherD}. Parameters: $\nu=0$, $R=5$, $y=(10,0,\cdots,0)$.}\label{tab:2}
\end{table}

To investigate the effect of the distributional order of the 'source,' we set $j=1$ and $\nu=5$ in~\eqref{eqn:Helm}; that is, now $f=\partial_1^{5}\delta_y$. Figure~\ref{fig:hiorder} shows the estimate from~\eqref{eqn:guarantee} and the values $|\widehat{U}^{\mathcal{C}}_m|$ from~\eqref{eqn:direct} for $n=2,3,4,5$, $m=100$, $k=2\pi$, $R=5$, and $y=(10,0,\cdots,0)$. Table~\ref{tab:hiorder} lists the predicted and the actual bandwidths.
\begin{figure}[hbt]
\hspace{0cm}\includegraphics[scale=0.6]{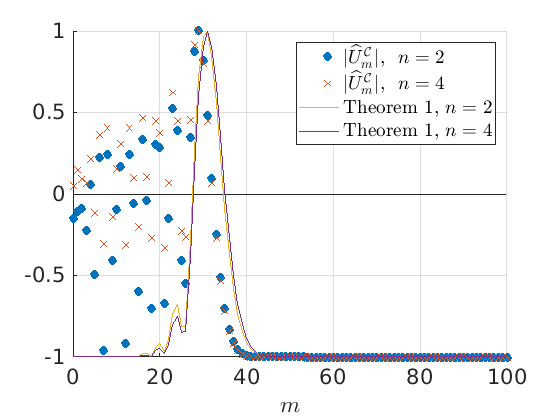}\includegraphics[scale=0.6]{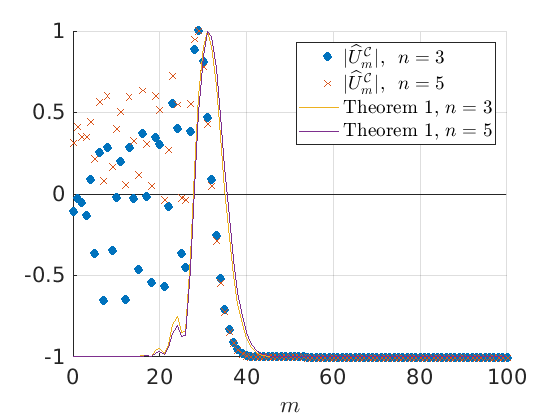}
\caption{Comparison of the actual spectrum $|\widehat{U}^{\mathcal C}_m|$ with the bound from Theorem~\ref{thm:MAIN}. The spectrum and the bound are shifted and scaled to range in $[-1,1]$. We are interested in the spectral location of the onset of rapid decay of $|\widehat{U}^{\mathcal{C}}_m|$. Parameters: $\nu=5$, $R=5$, $y=(10,0,\cdots,0)$.}\label{fig:hiorder}
\end{figure}
\begin{table}[hbt]
\begin{center}
\begin{tabular}{ |c|c|c| } 
 \hline
 dimension $n$ & bandwidth predicted & actual \\
               & by Theorem~\ref{thm:MAIN} & bandwidth\\
 \hline
 2 & 31 & 29 \\ 
 3 & 31 & 29 \\ 
 4 & 31 & 29 \\
 5 & 31 & 29 \\
 \hline
\end{tabular}
\end{center}
\caption{Predicted and measured bandwidths of the spectrum $\widehat{U}^{\mathcal{C}}_m$ from Figure~\ref{fig:hiorder}. Parameters: $\nu=5$, $R=5$, $y=(10,0,\cdots,0)$.}\label{tab:hiorder}
\end{table}

Our numerical investigation has shown that choosing $|y|\approx R$ and a low-order source in high dimension ($\nu=0$, $n\in\{6,7,8,9\}$) or a high-order source in low dimension ($\nu=5$, $n\in\{2,3,4,5\}$) results in a Fourier spectrum $\widehat{U}^{\mathcal{C}}_m$ that is quite different than what is shown in figures~\ref{fig:compar},~\ref{fig:higherD} and~\ref{fig:hiorder}, and not well-predicted by Theorem~\ref{thm:MAIN}. The prerequisites for the occurrence of this issue are consistent with the increasing severity of the singularity of the radiated field having an adverse effect on the numerical computations; however, we leave this to a future investigation. In all numerically investigated cases except when $|y|\approx R$ and $\nu=0$, $n\in\{6,7,8,9\}$ or $\nu=5$, $n\in\{2,3,4,5\}$, we found that Theorem~\ref{thm:MAIN} provides a valid upper bound on the spectral position of the onset of rapid decay for the Fourier spectrum of the measurement.

\subsection{Integrable compactly supported source}\label{subsec:integrable}

Next, we study the case where the inhomogeneity $f$ in
\[
(\Delta+k^2)u=f
\]
is a compactly supported 'source' in $L^1(\mathbf{R}^n)$. Clearly, this includes all $L^p$ sources with compact support, since $f\in\mathcal{E}'(\mathbf{R}^n)\cap L^p(\mathbf{R}^n)$ for any particular $p\in[1,\infty]$ implies $f\in L^p_{\rm loc}(\mathbf{R}^n)\subseteq L^1_{\rm loc}(\mathbf{R}^n)$, hence $f\in \mathcal{E}'^0(\mathbf{R}^n)\cap L^1(\mathbf{R}^n)$. Impose again an appropriate Sommerfeld radiation condition, and write $d(n)$ for the order of $\Phi_n$ as a tempered distribution, estimated in Lemma~\ref{lemma:H}. For any $\phi\in\mathcal{S}(\mathbf{R}^n)$ we have $f(-\cdot)\ast\phi\in\mathcal{S}(\mathbf{R}^n)$, so
\begin{align*}
|u(\phi)|&=|(f\ast\Phi_n)(\phi)|=|\Phi_n(f(-\cdot)\ast\phi)|\le C\sum_{|\alpha|\le d(n)}\langle x\rangle^{d(n)}\sup_{x\in\mathbf{R}^n}|f\ast\partial^{\alpha}\phi|\\&\le C\|f\|_{L^1(\mathbf{R}^n)}\sum_{|\alpha|\le d(n)}\langle x\rangle^{d(n)}\sup_{x\in\mathbf{R}^n}|\partial^{\alpha}\phi(x)|,
\end{align*}
and consequently $u\in\mathcal{S}'^{d(n)}(\mathbf{R}^n)$. Thus the same spectral cutoff estimates~\eqref{eqn:guarantee} apply as in Section~\ref{subsec:pointsource}.
\begin{figure}
\hspace{0cm}\includegraphics[scale=.5]{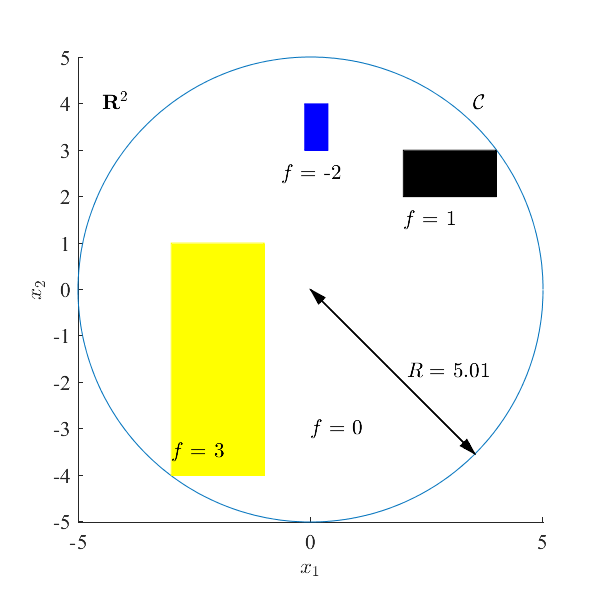}\includegraphics[scale=.6]{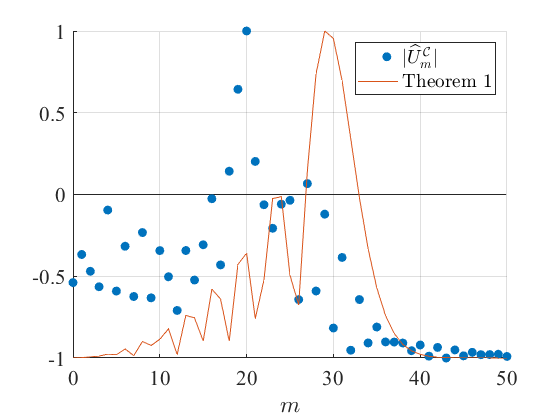}
\caption{Left: A setup with a compactly supported, piecewise constant source in $\mathbf{R}^2$. Right: Comparison of the actual spectrum $|\widehat{U}^{\mathcal{C}}_m|$ with the bound from Theorem~\ref{thm:MAIN}. The spectrum and the bound are shifted and scaled to range in $[-1,1]$. We are interested in the spectral location of the onset of rapid decay of $|\widehat{U}^{\mathcal{C}}_m|$.}\label{fig:pcsR2}
\end{figure}
Figure~\ref{fig:pcsR2} shows the estimates from~\eqref{eqn:guarantee}, as well as the numerically computed values of
\[
|\widehat{U}^{\mathcal{C}}_m|=\left|\int_{\theta=0}^{2\pi}e^{-im\theta}u(e_1R\cos\theta+e_2R\sin\theta)\right|,\quad m\in\{0,\dots,50\},
\]
with $n=2$, $k=2\pi$, $R=5.01$, $e_1=\widehat{x}_1$, $e_2=\widehat{x}_2$, and where the inhomogeneity $f$ is a piecewise constant, compactly supported function in $\mathbf{R}^2$, with $\supp f\subset\{|x|\le5\}$. In this case, the upper bound on the measurement bandwidth predicted by Theorem~\ref{thm:MAIN} is 29, and our best reading of the actual bandwidth is 27. Both the predicted and the measured bandwidths correspond well with our results in~\cite[Theorem 1 and Conjecture 1]{2018a-bandwidth}, where we estimate the bandwidth to be in the set
\[
\{{\rm argmin}_{m\in\mathbf{N}_0}\{j_{m,1}\ge kR\},\dots,{\rm argmin}_{m\in\mathbf{N}_0}\{y_{m,1}\ge kR\}\}=\{26,\dots,29\}
\]
for $kR=2\pi\cdot 5.01$. Here $j_{\mu,1}$ and $y_{\mu,1}$ are the first positive zeros of the Bessel function $J_{\mu}$ and the Neumann function $Y_{\mu}$, respectively. Our findings are also consistent with Griesmaier et al. (2014) and Griesmaier and Sylvester (2017), who estimate that the singular values of the source-to-\textit{far}-field measurement operator for the Helmholtz equation in $\mathbf{R}^2$ in this case decay rapidly when $|m|\ge kR\approx 31$.

Finally, Figure~\ref{fig:pcsR3} on page~\pageref{fig:pcsR3} shows the estimates and the actual spectrum for $n=3$, $k=2\pi$, $R=1.01$, and where the inhomogeneity $f$ is a piecewise constant, compactly supported function in $\mathbf{R}^3$, with $\supp f\subset\{|x|\le1\}$.
\begin{figure}
\hspace{0cm}\includegraphics[scale=.6]{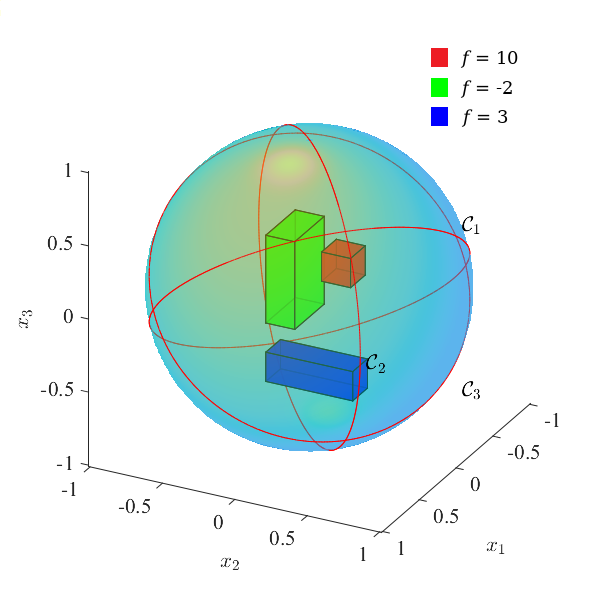}\includegraphics[scale=.6]{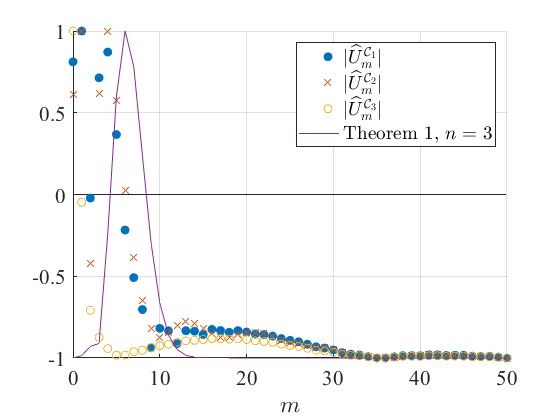}
\caption{Left: A setup with a compactly supported, piecewise constant source in $\mathbf{R}^3$. Right: Comparison of the actual spectrum $|\widehat{U}^{\mathcal{C}}_m|$ with the bound from Theorem~\ref{thm:MAIN}. The spectra and the bound are shifted and scaled to range in $[-1,1]$. We are interested in the spectral location of the onset of rapid decay of $|\widehat{U}^{\mathcal{C}_j}_m|$ for $j=1,2,3$.}\label{fig:pcsR3}
\end{figure}
We show the spectra for $u|_{\mathcal{C}}$ sampled over three different great circles of $S^2$, one of which ($\mathcal{C}_2$) is in near proximity of $\supp f$. While the bandwidth bound of Theorem~\ref{thm:MAIN} for this problem is 6, we read the actual bandwidths as 4, 4 and zero, for measurement over $\mathcal{C}_1$, $\mathcal{C}_2$ and $\mathcal{C}_3$, respectively. It is unsurprising that the bandwidth for measurement at $\mathcal{C}_3$ is extremely low, as this great circle is far removed from the support of the source $f$ and the trace of the field $u$ at $\mathcal{C}_3$ cannot feature rapid variation. In~\cite[Theorem 2]{2020a-Helmholtz_3D} we derived a lower bound on the bandwidth of the measurement in Figure~\ref{fig:pcsR3} over the whole $S^2$. This lower bound, and an associated conjectured upper bound, are in terms of the expansion of the measurement in spherical harmonics, and are given by
\[
\{{\rm argmin}_{m\in\mathbf{N}_0}\{j_{m+1/2,1}\ge kR\},\dots,{\rm argmin}_{m\in\mathbf{N}_0}\{y_{m+1/2,1}\ge kR\}\}=\{3,\dots,5\}
\]
for $kR=2\pi\cdot 1.01$. The bandwidths hence appear similar for measurements over the whole $S^2$ and along single great circles of $S^2$.

\section{Conclusion}\label{section:conclusion}

We proved Theorem~\ref{thm:MAIN} as a way to justify and compute estimates of the spectral position for the onset of rapid Fourier spectrum decay for traces of tempered solutions of a large class of multiplier equations in $\mathbf{R}^n$. The estimates were uniform for solutions up to a given order. In the process, we constructed in Lemma~\ref{lemma:particul} particular solutions of the multiplier equations and characterized in Theorem~\ref{thm:homogen} all solutions of the homogeneous multiplier equations. In Lemma~\ref{lemma:osobina} we found a rather explicit expression, Eq.~\eqref{eqn:p_inverse}, for a tempered fundamental solution of a multiplier. Finally, we successfully verified the spectral width estimates of Theorem~\ref{thm:MAIN} against numerically computed measurement spectra in several scenarios involving the inhomogeneous Helmholtz equation in $\mathbf{R}^n$ with $n=1,\dots,9$. Further investigation may be needed, however, of cases where the radiated field has severe singularities close to the measurement set, that is, to the great circle $\mathcal{C}$.

Theorem~\ref{thm:MAIN} indicates that the upper bound on the spectrum of the measurement may behave uniformly for a rather large class of multiplier equations. Interestingly, the Bessel functions appear in the estimate independently of the particular multiplier, and in fact arise from the geometry of the measurement set $\mathcal{C}$.

% \printbibliography
\bibliographystyle{amsplain}
\bibliography{loti_references}

\end{document}

%% file: setup.pdf_t
\begin{picture}(0,0)%
\includegraphics{setup.pdf}%
\end{picture}%
\setlength{\unitlength}{3158sp}%
\begingroup\makeatletter\ifx\SetFigFont\undefined%
\gdef\SetFigFont#1#2#3#4#5{%
  \reset@font\fontsize{#1}{#2pt}%
  \fontfamily{#3}\fontseries{#4}\fontshape{#5}%
  \selectfont}%
\fi\endgroup%
\begin{picture}(3545,3091)(1711,-3144)
\put(1726,-2911){\makebox(0,0)[lb]{\smash{{\SetFigFont{10}{12.0}{\familydefault}{\mddefault}{\updefault}{\color[rgb]{0,0,0}$u\in\mathcal{S}'(\mathbf{R}^n)$}%
}}}}
\put(1726,-2611){\makebox(0,0)[lb]{\smash{{\SetFigFont{10}{12.0}{\familydefault}{\mddefault}{\updefault}{\color[rgb]{0,0,0}$Pu=f\in\mathcal{E}'(\mathbf{R}^n)$}%
}}}}
\put(4112,-1511){\makebox(0,0)[lb]{\smash{{\SetFigFont{10}{12.0}{\familydefault}{\mddefault}{\updefault}{\color[rgb]{0,0,0}$e_1$}%
}}}}
\put(3612,-2245){\makebox(0,0)[lb]{\smash{{\SetFigFont{10}{12.0}{\familydefault}{\mddefault}{\updefault}{\color[rgb]{0,0,0}$R$}%
}}}}
\put(5241,-2605){\makebox(0,0)[lb]{\smash{{\SetFigFont{10}{12.0}{\familydefault}{\mddefault}{\updefault}{\color[rgb]{0,0,0}$U^{\mathcal{C}}=u|_{\mathcal C}$}%
}}}}
\put(3954,-212){\makebox(0,0)[lb]{\smash{{\SetFigFont{10}{12.0}{\familydefault}{\mddefault}{\updefault}{\color[rgb]{0,0,0}$\mathcal{C}=\{e_1R\cos\theta+e_2R\sin\theta,\,\theta\in[0,2\pi]\}$}%
}}}}
\put(4182,-1977){\makebox(0,0)[lb]{\smash{{\SetFigFont{10}{12.0}{\familydefault}{\mddefault}{\updefault}{\color[rgb]{0,0,0}$e_2$}%
}}}}
\put(2101,-1036){\makebox(0,0)[lb]{\smash{{\SetFigFont{10}{12.0}{\familydefault}{\mddefault}{\updefault}{\color[rgb]{0,0,0}sing supp $f$}%
}}}}
\put(4876,-1111){\makebox(0,0)[lb]{\smash{{\SetFigFont{10}{12.0}{\familydefault}{\mddefault}{\updefault}{\color[rgb]{0,0,0}$\mathcal{C}\,\,\cap$ sing supp $f$ = $\emptyset$}%
}}}}
\put(4876,-1411){\makebox(0,0)[lb]{\smash{{\SetFigFont{10}{12.0}{\familydefault}{\mddefault}{\updefault}{\color[rgb]{0,0,0}$\mathcal{C}\,\,\cap$ supp $f\neq\emptyset$ in general}%
}}}}
\end{picture}%